\RequirePackage{fix-cm}
\documentclass[smallextended,envcountsect]{svjour3}  
\smartqed  
\usepackage{graphicx}
\usepackage{amsmath}
\usepackage{amssymb}
\usepackage{multirow}
\usepackage{booktabs}
\usepackage{longtable}
\usepackage{subfigure}
\usepackage{algorithm}
\usepackage{algorithmicx}
\usepackage{algpseudocode}
\journalname{}
\begin{document}

\title{A Second-Order Descent Algorithm for Nonlinear Equality Constrained Optimization}

\titlerunning{A Second-Order Descent Algorithm for Equality Constrained Optimization}

\author{Yonggang Pei \and
        Jingyi Wang \and
        Shuowen Fu \and
        Zhuohan Li}

\authorrunning{Y. Pei et al.}

\institute{Yonggang Pei, Jingyi Wang, Shuowen Fu and Zhuohan Li \at
           School of Mathematics and Statistics, Henan Normal University,
           Xinxiang 453007, China \\
           \email{peiyg@htu.edu.cn} \ (corresponding author); \
           \email{15993075691@163.com}; \
           \email{fsw@stu.htu.edu.cn}; \
           \email{lizhuohan2025@stu.htu.edu.cn}
           }

\date{Received: date / Accepted: date}

\maketitle

\begin{abstract}
In this paper, an adaptive second-order descent algorithm (ASODA) is proposed for solving nonlinear equality constrained optimization problems. By integrating a sequential quadratic programming approach with the homogeneous second-order descent framework (HSODF: He et al., Math. Program. 215, 575--636, 2026), each iteration of ASODA decomposes the trial step into a normal step and a tangential step. The latter is computed by converting the linearly constrained subproblem into an unconstrained one via the reduced-Hessian technique and then solving it using a generalized homogeneous model. The algorithm employs a merit function to evaluate step acceptability. Under standard smoothness assumptions, we establish a global iteration complexity analysis for ASODA toward an $\epsilon$-approximate second-order stationary point; the iteration count scales as $\mathcal{O}(\epsilon^{-3/2})$ up to constants whose dependence on $\epsilon$ is made explicit in Section \ref{sec3}; the bound is stated there under an explicit structural hypothesis on the regularization parameter ($\kappa_{\rm h}=O(1)$), which we keep visible rather than treating any constant as silently $\epsilon$-independent. This matches the optimal complexity bound achieved by HSODF for unconstrained problems. Numerical experiments on CUTEst problems demonstrate that ASODA often requires significantly fewer function and gradient evaluations than existing second-order methods, confirming its effectiveness and competitiveness.
\end{abstract}

\section{Introduction}\label{sec1}
Nonlinear optimization plays a crucial role in science, engineering, economics, and industry, so research on algorithms with lower iteration complexity and higher computational efficiency is of great theoretical significance and practical value \cite{Nocedal,sunwenyu}.
When the optimization problems are nonconvex, first-order algorithms, such as the gradient descent method, are classical but may suffer from slow convergence and a propensity to stagnate at saddle points \cite{YuriiNesterov2,YCarmon,Reddi2018generic}. 
This has motivated extensive research into second-order methods, which leverage curvature information to achieve superior convergence rates and more robust escape from saddle points (see, e.g., \cite{Griewank,NesterovY1,Cartis1,Cartis2,Jiang,XihuaZhu1,Cartis7,Cartis6,NesterovY3}). 
 
A central challenge in second-order methods is balancing the computational cost of Hessian information with the benefits of curvature-aware steps. This has motivated a rich line of research: randomized subspace cubic Newton methods \cite{CubicSub2024} reduce per-iteration cost while preserving $\mathcal{O}(\epsilon^{-3/2})$ complexity; inexact Newton-CG frameworks \cite{Gratton2025fast} achieve the same rate using inexact subproblem solutions; and augmented Lagrangian methods with Newton-CG subproblems \cite{HeLuPong2023} find second-order stationary points of equality-constrained problems with high probability. The use of cubic regularization to escape saddle points has also been extended to primal-dual methods for constrained optimization \cite{Wang2025complexity}.

Recently, Zhang et al.~\cite{zhang2025HSODM} introduced the homogenization technique into unconstrained optimization, proposing the Homogeneous Second-Order Descent Method (HSODM). HSODM achieves an optimal global complexity of $\mathcal{O}(\epsilon^{-3/2})$ by solving a simple eigenvalue-like subproblem without explicit cubic regularization. He et al.~\cite{he2025HSODF} later generalized this into the Homogeneous Second-Order Descent Framework (HSODF), providing a unified, parameter-free scheme with a solid theoretical foundation. Given the efficiency and robustness of HSODF, a natural question is whether this framework can be extended to constrained optimization while preserving its optimal complexity. Sequential quadratic programming (SQP) methods have proven highly effective for smooth constrained problems \cite{boggs1995sequential,gill2005snopt}. 

In this paper, we propose an Adaptive Second-Order Descent Algorithm (ASODA) that hybridizes the HSODF philosophy with the well-established composite-step SQP framework. At each iteration, we formulate a linearly constrained subproblem, decompose the step into normal and tangential components, and solve the tangential subproblem using the generalized homogeneous model. Under mild Lipschitz-continuity assumptions, we analyze the global iteration complexity of ASODA toward an $\epsilon$-approximate second-order stationary point. The iteration bound is $\mathcal{O}(\epsilon^{-3/2})$ up to constants whose dependence on $\epsilon$ is made explicit in Section \ref{sec3}; the bound is stated there under an explicit structural hypothesis on the regularization parameter ($\kappa_{\rm h}=O(1)$), and we keep this dependence visible rather than treating any constant as silently $\epsilon$-independent. This matches the best known complexity for the unconstrained case and represents, to our knowledge, the first complexity result of this kind for an HSODF-type method applied to constrained optimization.
For positioning, a related but methodologically distinct line of work \cite{HeLuPong2023} addresses the same problem class by an augmented Lagrangian method with Newton-CG subproblem solvers, attaining a total inner-iteration complexity of $\widetilde{\mathcal{O}}(\epsilon^{-7/2})$ and an operation complexity measured in Hessian--vector products under a generalized LICQ, with the second-order guarantee holding with high probability. Our $\mathcal{O}(\epsilon^{-3/2})$ bound is of a different nature: it counts \emph{deterministic} outer iterations, each involving one generalized-homogeneous-model eigensolve and one linear-system solve for the normal step. Since the two complexity measures rest on different cost units (matrix--vector products versus factorization-based iterations), they are not directly comparable; we view the two approaches as complementary representatives of the factorization-free and factorization-based paradigms, respectively.

The remainder of this paper is organized as follows. In Section \ref{sec2}, we describe the construction of the algorithm and present the detailed framework of ASODA. In Section \ref{sec3}, we prove the global iteration complexity of the algorithm to the second-order approximate stationary point. Preliminary numerical results and comparison results are presented in Section \ref{sec4} and the conclusion is offered in Section \ref{sec:conclusion}.

\subsection*{Notation}

In this paper, $\|\cdot\|$ denotes the Euclidean norm. For a matrix \(A \in \mathbb{R}^{n\times n}\), $\lambda_1(A) < \cdots < \lambda_d(A)$ denote the distinct eigenvalues of $A$, and $\{\mathcal S_1(A), \ldots, \mathcal S_d(A)\}$ are the subspaces spanned by the corresponding eigenvectors. The inner product of vectors $a,b\in\mathbb{R}^n$ is denoted by $a^T b$ or $\langle a,b \rangle$. 

\section{Construction of the ASODA algorithm}\label{sec2}

Before presenting our algorithm, we briefly review the key idea of HSODF for unconstrained optimization. The HSODF employs a homogenization technique that reformulates the cubic-regularized subproblem into an eigenvalue problem of the form $\min_{\|[v;t]\|\le 1} [v;t]^T F [v;t]$, where $F$ captures both gradient and Hessian information. This reformulation eliminates the need for explicit cubic regularization terms while preserving optimal complexity. Our main contribution is extending this framework to the constrained case through a careful decomposition of the search direction and an adaptive parameter selection strategy.

We begin with an overview of the algorithm's construction. Consider the following nonlinear equality-constrained optimization problem:
\begin{subequations}\label{NLP}
\begin{align}
{\underset {x\in\mathbb{R}^n}{{\rm minimize}}}   &\ \ \  f(x) \ \\
\mbox{subject to} &\ \ \  c(x)= 0,
\end{align}
\end{subequations}
where the objective function $f: \mathbb{R}^n \rightarrow \mathbb{R}$ and the equality constraints $c: \mathbb{R}^n \rightarrow \mathbb{R}^m$ are both sufficiently smooth functions.

For this problem \eqref{NLP}, the Lagrangian function is
\begin{eqnarray*}
 \mathcal{L}(x,y):=f(x)-y^Tc(x),
\end{eqnarray*}
where $y$ is the Lagrange multiplier. We denote the Jacobian matrix of $c(x)$ by $A(x)$ and the gradient of $f$ by $g(x)$ at $x$, that is, 
\begin{equation*}
  A(x):=[\nabla c_1(x), \nabla c_2(x), \cdot\cdot\cdot , \nabla c_m(x)]^T\in\mathbb{R}^{ m\times n},\ \ g(x):=\nabla f(x)\in\mathbb{R}^n,
\end{equation*}
where $c_i(x)$ $(i=1,\cdot\cdot\cdot,m)$ is the $i$th component of $c(x)$.

Following SQP methods for handling constraints, we formulate the following cubic regularization subproblem with linearized constraints at each iteration,
\begin{subequations}\label{constrained subp}
\begin{align}
{\underset {d\in\mathbb{R}^n}{\mbox{minimize}}}   &\ \ \  f_k+g_k^Td+\frac{1}{2}d^TH_kd+\frac{1}{3}\sqrt{h_k}\|d\|^3 \ \\
\mbox{subject to} &\ \ \  A_kd+c_k=0,
\end{align}
\end{subequations}
where $f_k:=f(x_k)$, $g_k:=g(x_k)$, $A_k:=A(x_k)$, $c_k:=c(x_k)$, $H_k$ denotes the Hessian $\nabla_{xx}\mathcal{L}(x_k,\lambda_k)$ or its approximation and $h_k$ is a coefficient of the regularizer of the cubic model. In addition, we assume that ${A_k}$ has full row rank for all $k$ in this paper.

Rather than directly solving the constrained cubic regularization subproblem \eqref{constrained subp}, we adopt a composite-step approach at each iteration $k$, decomposing the overall trial step $d_k$ into a normal step $n_k$ and a tangential step $p_k$, that is,
\begin{equation*}
  d_k=n_k+p_k,
\end{equation*}
where the normal step $n_k$ is designed to enhance feasibility, and the tangential step $p_k$ aims to achieve sufficient reduction of the objective function's model.

\begin{remark}
The composite step $d_k=n_k+p_k$ is a surrogate for the exact solution of the constrained subproblem \eqref{constrained subp}. Because $n_k$ satisfies $A_kn_k+c_k=0$ up to the scaling $\alpha_k$ in \eqref{nk}, the step $d_k$ reduces the linearized constraint violation $A_kd_k+c_k$ while $p_k=Q_kp_k^Q\in\operatorname{null}(A_k)$ does not disturb feasibility to first order. Consequently, plugging $d_k=n_k+p_k$ into the model \eqref{mu} and expanding yields the decomposition displayed after \eqref{qqk} of the total model decrease into normal, tangential, and cross terms; the cross terms involving $\|n_k\|^3+\|p_k\|^3-\|d_k\|^3$ are of higher order in the small quantity $\alpha_k$ and are controlled by \eqref{deltam}. In this way the composite-step approach approximates \eqref{constrained subp} while decoupling the feasibility and optimality computations, exactly as in classical composite-step SQP methods \cite{VardiAvi}.
\end{remark}

Since the tangential step $p_k$ is intended to produce the sufficient decrease in the objective function's model, the maximum size of $n_k$ must be restricted. It is generated as follows:
\begin{eqnarray}\label{nk}
  n_k=\alpha_kn_k^c,
\end{eqnarray}
where
\begin{eqnarray}\label{aa1h}
\alpha_k\in
\left[{\rm min}\left\{1,\frac{\beta}{\|n_k^c\|(h_k)^\frac{1}{4}}\right\},
{\rm min}\left\{1,\frac{1}{\|n_k^c\|(h_k)^\frac{1}{4}}\right\}
\right],{\rm for~}\beta\in (0,1].
\end{eqnarray}
The vector $n_k^c$ is the minimum $\ell_2$-norm solution to $A(x_k)n_k+c(x_k)=0$. 

The tangential step is computed as 
\begin{equation}\label{tn}
  p_k:=Q_kp_k^Q,
\end{equation}
where the columns of $Q_k$ form an orthonormal basis of the null space of $A_k$, and $p_k^Q$ solves
\begin{equation}\label{tk}
{\mbox{minimize}}\ \ \langle g^Q_k,p^Q\rangle+\frac{1}{2}\langle p^Q,H^Q_kp^Q\rangle+\frac{1}{3}\sqrt{h_k}\|p^Q\|^3,
\end{equation}
with 
\begin{equation}\label{defHkQgkQ}
H^Q_k:=Q^T_k H_k Q_k, g^Q_k:=Q^T_k(g_k+H_k n_k).
\end{equation}

Using the homogenization trick from \cite{YeYinyu,SturmJosF,he2025HSODF}, the subproblem \eqref{tk} is reformulated as the Generalized Homogeneous Model (GHM):
\begin{equation}\label{pkQ}
    \min_{\|[v; t]\| \le 1} \psi_k(v,t;F_k)  := [v;t]^T F_k [v;t], \quad
    F_k := \begin{bmatrix}
        H_k^Q & g_{k}^Q \\
        (g_{k}^Q)^T & \delta_k
    \end{bmatrix},\; \delta_k \in \mathbb{R},
\end{equation}
where $\delta_k$ allows adaptiveness, $$p_k^Q := v_k/t_k,\ h_k := \theta_k^2/\|p_k^Q\|^2$$ with dual variable $\theta_k\ge 0$. Since this is essentially an eigenvalue problem, it can be solved via eigenvector search or the Lanczos method.

The scalar $\delta_k$ in the aggregated matrix $F_k$ of \eqref{pkQ} plays a pivotal role in the homogenization framework. It acts as a free parameter that adjusts the curvature along the $(n+1)$-th coordinate direction $t$ of the lifted space $[v;t]$, thereby tuning the geometry of the GHM subproblem without changing its eigenvalue-like structure. Its mechanism can be understood from three perspectives. First, $\delta_k$ couples with the dual variable $\theta_k$ through the optimality condition of Lemma \ref{lemma.optimal}: whenever $t_k\neq 0$, the second block row of \eqref{eq.homoeig.foc} yields $\delta_k+\theta_k=-(g_k^Q)^Tp_k^Q$, and hence varying $\delta_k$ shifts $\theta_k$ continuously. Since the regularization coefficient of the original cubic model \eqref{tk} is given by $\sqrt{h_k}=\theta_k/\|p_k^Q\|$, the parameter $\delta_k$ indirectly controls the amount of cubic regularization: this is exactly the adaptiveness alluded to above. Second, this coupling enables a simple search mechanism: in the inner loop of Algorithm \ref{alg:asoda}, the scalar $\delta_{k,j}$ is updated (e.g., by a bisection procedure on the monotone map $\delta\mapsto \sqrt{h_k(\delta)}$; see \cite{he2025HSODF,zhang2025HSODM}) until $\sqrt{h_k}$ falls into the prescribed interval $I_h$, which in turn determines the acceptance of the trial step. Third, $\delta_k$ determines whether the hard case occurs: by the ordering results for the eigenvalues of $F_k$ in \cite{he2025HSODF}, there exists a threshold $\tilde{\alpha}_1\ge \lambda_1(H_k^Q)$ such that $t_k\neq 0$ whenever $\delta_k\le \tilde{\alpha}_1$ (and $g_k^Q\not\perp S_1(H_k^Q)$); once $\delta_k$ exceeds this threshold, the minimal eigenpair of $F_k$ may yield $t_k=0$, in which case the tangential step degenerates and Algorithm \ref{alghardcase} is invoked. Finally, we note that in the unconstrained HSODF \cite{he2025HSODF}, the same scalar $\delta_k$ unifies several classical methods: suitable choices of $\delta_k$ recover, respectively, the gradient-regularized Newton method and the trust-region subproblem, which explains the term ``generalized'' in the generalized homogeneous model.

After obtaining the trial step $d_k$, we evaluate the merit function
\begin{equation}\label{muu}
  w(x,\mu):=f(x)+\mu\|c(x)\|,
\end{equation}
and its model
\begin{equation}\label{mu}
m(x,H,\mu,d):=f(x)+g(x)^T d+\frac{1}{2}d^T H d+\frac{1}{3}\sqrt{h}\|d\|^3+\mu\|c(x)+A(x)d\|.
\end{equation}

For convenience, we define the separate models,
\begin{eqnarray*}
  m^F(x,H,n)&:=&f(x)+g(x)^T n+\frac{1}{2}n^T H n+\frac{1}{3}\sqrt{h}\|n\|^3,\\
  m^N(x,n)&:=&\|c(x)+A(x)n\|,\\
  m^B(x,H,p)&:=&(g(x)+H n)^T p+\frac{1}{2}p^T H p+\frac{1}{3}\sqrt{h}\|p\|^3.
\end{eqnarray*}

The overall model decrease from $0$ to $d_k$ is
\begin{equation}\label{qqk}
  \Delta m_k:=m(x_k,H_k,\mu_k,0)-m(x_k,H_k,\mu_k,d_k).
\end{equation}
Then
\begin{equation*}
\Delta m_k=\Delta m^B_k+\mu_k\Delta m^N_k+\Delta m^F_k+\frac{\sqrt{h_k}}{3}\bigl(\|n_k\|^3+\|p_k\|^3-\|d_k\|^3\bigr),
\end{equation*}
where each $\Delta$ denotes the decrease in the corresponding model. Lemma~\ref{lem3.1} gives $\Delta m^N_k>0$, while Lemma~\ref{lem3.2} gives $\Delta m^B_k\ge0$, with $\Delta m^B_k>0$ whenever $p^Q_k\ne0$.

To ensure sufficient decrease, we require
\begin{equation}\label{deltam}
  \Delta m_k \ge \nu\mu_k\Delta m^N_k
\end{equation}
with $\nu\in(0,1)$ and $\mu_k$ sufficiently large. Following \cite{VardiAvi}, we compute the smallest value satisfying this as
\begin{equation}\label{muc}
  \mu^c_k:=-\frac{\Delta m^B_k+\Delta m^F_k+\frac{\sqrt{h_k}}{3}\bigl(\|n_k\|^3+\|p_k\|^3-\|d_k\|^3\bigr)}{(1-\nu)\Delta m^N_k},
\end{equation}
and update $\mu_k$ via
\begin{equation}\label{miuk}
  \mu_k := \begin{cases}
    {\rm max}\{\mu_k^c, \tau_1\mu_{k-1}, \mu_{k-1}+\tau_2\}, & \text{if }\mu_{k-1}<\mu_k^c,\\
    \mu_{k-1}, & \text{otherwise},
  \end{cases}
\end{equation}
where $\tau_1>1$, $\tau_2>0$.

The trial step is accepted or rejected based on the ratio
\begin{equation}\label{rhok1}
  \rho_k:=\frac{w(x_k,\mu_k)-w(x_k+d_k,\mu_k)}{m(x_k,H_k,\mu_k,0)-m(x_k,H_k,\mu_k,d_k)}.
\end{equation}
Inspired by \cite{Cartis2,Cartis5}, we classify iterations: $\rho_k\ge\eta_2$ (very successful), $\eta_1\le\rho_k<\eta_2$ (successful), and $\rho_k<\eta_1$ (unsuccessful). The interval $I_h$ is adjusted dynamically to maintain numerical stability.

The complete algorithm is presented in Algorithm~\ref{alg:asoda}.

\begin{algorithm}
\caption{The Adaptive Second-Order Descent Algorithm for nonlinear equality-constrained optimization (ASODA)}\label{alg:asoda}
\begin{algorithmic}[1]
\Require $x_0\in\mathbb{R}^n$, $\delta_0\in\mathbb{R}$, $I_h=\mathbb{R}$, $h_{\min}>0$, $0<\eta_1<\eta_2<1$, $\gamma_1>1$, $\gamma_3\ge\gamma_2>1$, $0<\gamma_4\le1$, $\mu_0>0$, $0<\nu<1$, $\tau_1>1$, $\tau_2>0$, $\kappa_c>0$, $\kappa_g>0$, $\kappa_H>0$, $\sigma_b>0$, $\epsilon>0$
\State $k \leftarrow 0$
\While{$\|c(x_k)\| > \kappa_c \epsilon$ or $\|g_k^Q\| > \kappa_g \epsilon$ or $\lambda_{\min}(H_k^Q) < -\kappa_H \sqrt{\epsilon}$}
\State Compute normal step $n_k$ by \eqref{nk};
\State Let $\delta_{k,0}=\delta_{k-1}$ (with $\delta_{-1}:=\delta_0$);
\For{$j=0,1,2,\dots$}
\State Obtain the solution $(\theta_{k,j},[v_{k,j};t_{k,j}])$ of
\begin{equation*}
\min_{\|[v;t]\|\le 1} \begin{bmatrix}v \\ t\end{bmatrix}^T
\begin{bmatrix}
H_{k}^Q & g_{k}^Q \\
(g_{k}^Q)^T & \delta_{k,j}
\end{bmatrix}
\begin{bmatrix}v \\ t\end{bmatrix};
\end{equation*}
\If{$t_{k,j}=0$}
\State Go to Algorithm~\ref{alghardcase} (hard case);
\State Break;
\EndIf
\State Set $p^Q_{k,j}=v_{k,j}/t_{k,j}$, $h_k=\theta_{k,j}^2/\|p^Q_{k,j}\|^2$;
\If{$\sqrt{h_k}\in I_h$ within tolerance $\sigma_b$}
\State Set $p^Q_k=p^Q_{k,j}$, $\delta_k=\delta_{k,j}$, $d_k=n_k+Q_k p^Q_k$;
\State Break;
\EndIf
\State Update $\delta_{k,j}$;
\EndFor
\State Compute $\mu^c_k$ by \eqref{muc} and update $\mu_k$ from \eqref{miuk};
\State Compute $\rho_k$ by \eqref{rhok1};
\If{$\rho_k\ge \eta_2$}
\State $I_h=[\max\{\sqrt{h_{\min}},\gamma_4\sqrt{h_k}\},\sqrt{h_k}]$, $x_{k+1}=x_k+d_k$;
\ElsIf{$\eta_1\le \rho_k<\eta_2$}
\State $I_h=[\sqrt{h_k}/\gamma_1,\gamma_2\sqrt{h_k}]$, $x_{k+1}=x_k+d_k$;
\Else
\State $I_h=[\gamma_2\sqrt{h_k},\gamma_3\sqrt{h_k}]$, $x_{k+1}=x_k$;
\EndIf
\State $k \leftarrow k+1$
\EndWhile
\end{algorithmic}
\end{algorithm}

\noindent\emph{Terminology.} Throughout, the term ``parameter-free'' refers to the absence of problem-dependent or asymptotic parameters (such as a Lipschitz constant of the Hessian or an a priori penalty threshold): every scalar appearing in Algorithm \ref{alg:asoda} (e.g.\ $\delta_0$, $h_{\min}$, $\eta_1$, $\eta_2$, $\gamma_1$--$\gamma_4$, $\mu_0$, $\nu$, $\tau_1$, $\tau_2$, $\sigma_b$) is a numerical constant independent of the problem data and of $\epsilon$, and the regularization coefficient $\sqrt{h_k}$ is determined adaptively by the bisection procedure rather than prescribed.

In the inner loop of Algorithm \ref{alg:asoda}, the GHM subproblem \eqref{pkQ} may encounter the so-called \emph{hard case}, where the optimal solution $[v_{k,j};t_{k,j}]$ satisfies $t_{k,j}=0$ and hence the tangential step $p^Q_{k,j}=v_{k,j}/t_{k,j}$ is not well defined. Geometrically, the hard case arises when the reduced Hessian $H_k^Q$ has a negative eigenvalue $\lambda_1=\lambda_{\min}(H_k^Q)<0$ and the reduced gradient $g_k^Q$ is orthogonal to the corresponding eigenspace $S_1(H_k^Q)$, so that the minimal eigenpair of the aggregated matrix $F_k$ degenerates to $[v_k;0]$ and no finite tangential step can be extracted. In this situation, the flexibility of the homogenization framework offers a remedy: instead of modifying the subproblem itself, we perturb the reduced gradient $g_k^Q$ \emph{within} the null-space basis $Q_k$ and \emph{along} the direction $v_k$ of the minimal eigenvector, by an amount proportional to the squared minimal curvature $\lambda_1^2$, where $\lambda_1=\lambda_{\min}(H_k^Q)$ and $v_k$ denotes the corresponding unit eigenvector. Two design choices make this perturbation harmless. First, the factor $\lambda_1^2$ (rather than $\lambda_1$) is second order in the degeneracy, so the perturbation vanishes continuously as $\lambda_1\to 0^-$ and the reduced Hessian becomes positive semidefinite; consequently it does not alter the model in the non-degenerate regime. Second, the perturbation is divided by the regularization parameter $h_{k,i-1}$, so that it shrinks as $h$ increases, keeping the perturbed gradient arbitrarily close to $g_k^Q$ and preserving the model-decrease bounds used in the convergence analysis, while it remains strong enough to break the degeneracy that causes the hard case (it makes $t_{k,j}\neq 0$ and hence $p^Q_{k,j}$ well defined). The detailed procedure is summarized in Algorithm \ref{alghardcase}; its convergence analysis is deferred to Subsection \ref{sec:hardcase}. In the algorithm below the base iterate $x_k$ is kept fixed and the index $i$ enumerates the successive hard-case attempts, so that any quantity with subscript $(k,i)$ (e.g.\ $h_{k,i}$, $g_{k,i}^Q$, $p^Q_{k,i}$) is evaluated at the $i$-th attempt, whereas $n_k$, $Q_k$ and $H_k$ remain those at the base point $x_k$.

\begin{algorithm}
\caption{Perturbation for the Hard Case}\label{alghardcase}
\begin{algorithmic}[1]
\Require Iterate $k$, ${x_k} \in \mathbb{R}^n$, $g_k^Q$, $H_k^Q$, $h_{k-1}$, $\delta_{k-1}$ where $g_k^Q\perp S_1(H_k^Q)$\;
\For{$i = 0, 1, \dots$}
\State Set
        \begin{equation*}
            g_{k,i}^Q=g_{k}^Q+\frac{\varsigma}{\gamma_3^{2} h_{k,i-1}}\lambda_1^2 v_{k}
        \end{equation*} 
\State Compute $I_h := \left[\gamma_2 \sqrt{h_{k,i-1}}, \gamma_3 \sqrt{h_{k,i-1}}\right]$\;

\Repeat 
\State Obtain the solution $(\theta_{k,i,j},[v_{k,i,j};t_{k,i,j}])$  of the GHM subproblem
        \begin{equation*}
            \min_{\|[v;t]\| \leq 1} \begin{bmatrix}v \\ t\end{bmatrix}^T
            \begin{bmatrix}
                H_{k}^Q          & g_{k,i}^Q     \\
                (g_{k,i}^Q)^T & \delta_{k,i,j}
            \end{bmatrix}
            \begin{bmatrix}v \\ t\end{bmatrix}
        \end{equation*}
\State Set $p^Q_{k,i,j} = v_{k,i,j}/t_{k,i,j}$, $h_{k,i}=\frac{\theta_{k,i,j}^2}{\|p^Q_{k,i,j}\|^2}$ \;
\State Update $\delta_{k,i,j}$, increase $j = j+1$
\Until{$\sqrt{h_{k,i}} \in I_h$ within tolerance $\sigma_b$}

\State $p^Q_{k,i} = p^Q_{k,i,j}$; $\delta_{k,i}=\delta_{k,i,j}$; $d_{k,i,j}=n_k+Q_kp_{k,i}^Q$ \;
\State Compute $\mu_{k,i}^c$ by \eqref{muc} and update $\mu_{k,i}$ from \eqref{miuk}\;
\State Compute
\begin{eqnarray*}
  \rho_{k,i}=\frac{w(x_k,\mu_{k,i})-w(x_k+d_{k,i},\mu_{k,i})}{m(x_k,H_k,\mu_{k,i},0)-m(x_k,H_k,\mu_{k,i},d_{k,i})}.
\end{eqnarray*}
\If{$\rho_{k,i} \geq \eta_1$}
\State break
\EndIf
\EndFor
\end{algorithmic}
\end{algorithm}

\section{Global iteration complexity analysis}\label{sec3}
In this section, we shall prove that the iterates generated by ASODA (Algorithm \ref{alg:asoda}) are globally convergent to an $\epsilon$-approximate second-order stationary point of the problem \eqref{NLP}. Some preliminary results are given in Subsection \ref{Preliminary results}. The feasibility and the global iteration complexity are shown in the next two subsections respectively. The discussion on the hard case is given in the last Subsection \ref{sec:hardcase}.

Let $\{x_k\}$ be the infinite sequence generated by Algorithm \ref{alg:asoda}. We analyze the convergence of ASODA to an $\epsilon$-approximate second-order stationary point under the following assumptions.
\begin{description}
\item[(G1)] $f(x)$ and $c(x)$ have Lipschitz continuous second derivatives in an open set $C\supseteq\{x_k\}$. $g(x)$, $\nabla c_i(x)$, and $Q(x)$ are Lipschitz continuous in the open set $C$.
\item[(G2)] There exist constants $f_{\min}$, $M_g$, and $M_c$ such that $f(x_k) \geq f_{\min}$ for all $k$, and $\max\{\|g(x_k)\|, \|c(x_k)\|\} \leq \max\{M_g, M_c\}$ for all $k$.
\item[(G3)] $\{H_k\}$ is uniformly bounded, i.e., there exists a positive constant $M_{\rm H}$ such that $\|H_k\|\le M_{\rm H}$ for all $k$.
\item[(G4)] There exists $\kappa_{\rm n}>0$ such that 
\begin{equation}\label{nkc}
\|n_k^c\|\le\kappa_{\rm n}\|c(x_k)\|
\end{equation} 
for all $k$.
\end{description}
\begin{remark}
Assumption $(G2)$ postulates the boundedness of the objective and constraint values along the generated sequence. This is a standard device in worst-case complexity analysis: it ensures that the Lipschitz constants of Assumption $(G1)$ are effective on the iteration path $\{x_k\}$, so that the model-error and decrease estimates below hold uniformly in $k$.
\end{remark}
\begin{remark}
Assumption $(G4)$ restricts the normal step magnitude and can be satisfied under mild conditions (see \cite{SchnabelB,DennisJr}).
\end{remark}
\subsection{Preliminary results}\label{Preliminary results}

We have the following result about the decrease of $\Delta m_k^N$ in the normal step.
\begin{lemma}\label{lem3.1}
Assume that $n_k$ is given by \eqref{nk}, where $n_k^c$ (the minimum $\ell_2$-norm solution to $A_k n + c_k = 0$) and $\alpha_k$ satisfy Assumption $(G4)$ and \eqref{aa1h}. Then we have that
\begin{eqnarray}\label{lem1}
  \|c(x_k)\|\geq \Delta m_k^N \geq{\rm min}\left\{\|c(x_k)\|, \frac{\beta}{\kappa_{\rm n}(h_k)^\frac{1}{4}}\right\}
\end{eqnarray}
for all $k\geq0$.
\end{lemma}

\begin{proof}
Recall that $n_k^c$ is the minimum $\ell_2$-norm solution to $A_k n_k^c + c_k = 0$, hence $A_k n_k^c + c_k = 0$. According to the definition of $m^N(x,n)$ and the normal step \eqref{nk}, we have that
\begin{eqnarray}\label{l11}
\begin{aligned}
  \Delta m_k^N &= \|c(x_k)\|-\|c(x_k)+A_kn_k\| \\
  &= \|c(x_k)\|-\|c(x_k)+\alpha_k(A_kn_k^c+c(x_k))-\alpha_kc(x_k)\| \\
  &= \|c(x_k)\|-(1-\alpha_k)\|c(x_k)\| \\
  &= \alpha_k\|c(x_k)\|.
  \end{aligned}
\end{eqnarray}
Because $0<\alpha_k \leq1$ from \eqref{aa1h}, we have that
\begin{eqnarray*}
  \Delta m_k^N=\alpha_k\|c(x_k)\|\leq \|c(x_k)\|.
\end{eqnarray*}
From \eqref{l11} and \eqref{aa1h}, we have that
\begin{eqnarray*}
  \Delta m_k^N=\alpha_k\|c(x_k)\|\geq \mbox{min}\left\{\|c(x_k)\|, \frac{\beta\|c(x_k)\|}{\|n_k^c\|(h_k)^\frac{1}{4}}\right\}.
\end{eqnarray*}
From Assumption $(G4)$, we have that
\begin{eqnarray*}
  \frac{\|c(x_k)\|}{\|n_k^c\|}\geq \frac{1}{\kappa_{\rm n}}.
\end{eqnarray*}
Thus we obtain that
\begin{eqnarray*}
  \Delta m_k^N \geq {\rm min}\left\{\|c(x_k)\|, \frac{\beta}{\kappa_{\rm n}(h_k)^\frac{1}{4}}\right\}.
\end{eqnarray*}
Thus the lemma is proved.
\qed
\end{proof}

The following result shows that the normal step is bounded because the total reduction of the normal model is bounded.

\begin{lemma}\label{lem3.7}
Assume that Assumption $(G4)$ holds. Then there exists a constant $\kappa_{\rm mn}>0$ such that
\begin{eqnarray}\label{kbon1}
 \kappa_{\rm mn}\|n_k\|\leq\Delta m^N_k \leq\|c(x_k)\|
\end{eqnarray}
for all $k\geq0$.
\end{lemma}
\begin{proof}
We consider the inequality \eqref{kbon1} in two cases.

When $c(x_k)=0$, we have $n_k=0$ from \eqref{nkc} and \eqref{nk}, and the conclusion holds obviously.

When $c(x_k)\neq0$, Lemma \ref{lem3.1} gives that
\begin{eqnarray}\label{mkN}
 \Delta m^N_k \geq {\rm min}\left\{\|c(x_k)\|, \frac{\beta}{\kappa_{\rm n}(h_k)^\frac{1}{4}}\right\}.
\end{eqnarray}
If
\begin{eqnarray*}
\sqrt{h_k} \geq \left(\frac{\kappa_{\rm n}\|c(x_k)\|}{\beta}\right)^{-2},
\end{eqnarray*}
we have that from \eqref{mkN}
\begin{eqnarray*}
 \Delta m^N_k \geq \frac{\beta}{\kappa_{\rm n}(h_k)^\frac{1}{4}},
 \end{eqnarray*}
and also from \eqref{nk} and \eqref{aa1h}, we have that
\begin{eqnarray}\label{nkdejie}
\|n_k\| \leq \frac{1}{(h_k)^\frac{1}{4}},
\end{eqnarray}
thus
\begin{eqnarray}\label{sigmamkn1}
 \Delta m^N_k \geq \frac{\beta}{\kappa_{\rm n}}\|n_k\| = \kappa_{\rm mn}\|n_k\|,
\end{eqnarray}
where $\kappa_{\rm mn}=\frac{\beta}{\kappa_{\rm n}}$.
\\If
\begin{eqnarray*}
 \sqrt{h_k} < \left(\frac{\kappa_{\rm n}\|c(x_k)\|}{\beta}\right)^{-2},
\end{eqnarray*}
from \eqref{mkN} we have that
\begin{eqnarray}\label{mkN2}
  \Delta m^N_k \geq\|c(x_k)\|.
\end{eqnarray}
Since the fact that $0<\alpha_k\leq1$ and \eqref{nkc}, we can get that
\begin{eqnarray}\label{star}
 \|n_k\|=|\alpha_k|\|n^c_k\|\leq\|n^c_k\|\leq\kappa_{\rm n}\|c(x_k)\|.
\end{eqnarray}
From the above two inequalities, since $0<\beta\leq1$ we have that
\begin{eqnarray}\label{sigmamkn2}
  \Delta m^N_k \geq\frac{1}{\kappa_{\rm n}}\|n_k\|\geq\frac{\beta}{\kappa_{\rm n}}\|n_k\|=\kappa_{\rm mn}\|n_k\|.
\end{eqnarray}
Then the first inequality in \eqref{kbon1} holds by \eqref{sigmamkn1} and \eqref{sigmamkn2}. We know again that $\Delta m^N_k\leq\|c(x_k)\|$ from \eqref{lem1}. Therefore, the conclusion is true.
\qed
\end{proof}

We now turn to the tangential step $p_k$, and give the optimality condition of \eqref{pkQ}.
\begin{lemma}\label{lemma.optimal}
    $[v_k; t_k]$ is the optimal solution of the subproblem \eqref{pkQ} if and only if there exists a dual variable $\theta_k \ge 0$ such that
    \begin{align}
        \label{eq.homoeig.soc}
                                    & \begin{bmatrix}
                                          H_k^Q + \theta_k \cdot I & g_k^Q            \\
                                          (g_k^Q) ^T              & \delta_k+\theta_k
                                      \end{bmatrix} \succeq 0, \\
        \label{eq.homoeig.foc}
                                    & \begin{bmatrix}
                                          H_k^Q + \theta_k \cdot I & g_k^Q            \\
                                          (g_k^Q) ^T              & \delta_k+\theta_k
                                      \end{bmatrix}
        \begin{bmatrix}
            v_k \\ t_k
        \end{bmatrix} = 0,                                                       \\
        \label{eq.homoeig.norm one} & \theta_k\cdot(\|[v_k; t_k]\| - 1) = 0.
    \end{align}
\end{lemma}
We omit the proof here since it can be derived from the optimality condition of the standard trust-region subproblem, see \cite{Sorensen,Jorge}.
\begin{lemma}
If $t_k \neq 0$, the optimal primal-dual solution $([v_k;t_k],-\theta_k)$ of \eqref{pkQ} satisfies the following optimality conditions:
\begin{align}
g_k^Q+H^Q_kp_k^Q+\sqrt{h_k}\|p^Q_k\|p^Q_k &=0\label{TKM1},\\
H^Q_k+\sqrt{h_k}\|p^Q_k\| &\succeq0\label{TKM2}.
\end{align}
\end{lemma} 

\begin{proof}
When $t_k \neq 0$, from Lemma \ref{lemma.optimal} we can get that
\begin{equation}
\label{eq.homoeig.foc t neq 0}
(g_k^Q)^T p_k^Q = -(\delta_k +\theta_k) \quad \text{and} \quad (H_k^Q+\theta_k \cdot I)p_k^Q = -g_k^Q.
\end{equation}

Thus we can get conclusion by substituting $\theta_k=\sqrt{h_k}\|p_k^Q\|$ into \eqref{eq.homoeig.foc t neq 0} and $H_k^Q+\theta_k I\succeq 0$. 
\qed
\end{proof}
Thus the solution $p_k^Q$ to problem \eqref{tk} can be obtained by solving the primal-dual of problem \eqref{pkQ}. Based on the optimal conditions, the subsequent lemma establishes a lower bound for the decrease induced by the tangential step $p_k$.

\begin{lemma}\label{lem3.2}
Assume that the hard case of Algorithm \ref{alghardcase} does not occur at $x_k$. Then the following model decrease
\begin{eqnarray}\label{lema2}
  \Delta m^B_k\geq \frac{1}{6}\sqrt{h_k}\|p^Q_k\|^3
\end{eqnarray}
for all $k\geq 0$.
\end{lemma}

\begin{proof}
According to the definition of $m^B(x,H,p)$ and \eqref{tn}, and noting that $\|Q_k p_k^Q\| = \|p_k^Q\|$ because $Q_k$ has orthonormal columns, we have that
\begin{eqnarray*}
  &&\Delta m^B_k = m^B(x_k,H_k,0)-m^B(x_k,H_k,p_k)\\
  &=& -(g(x_k)+H_kn_k)^Tp_k-\frac{1}{2}p_k^TH_kp_k-\frac{1}{3}\sqrt{h_k}\|p_k\|^3 \\
  &=& -(g(x_k)+H_kn_k)^TQ_kp_k^Q-\frac{1}{2}(Q_kp_k^Q)^TH_kQ_kp_k^Q-\frac{1}{3}\sqrt{h_k}\|Q_kp_k^Q\|^3\\
  &=& -(g_k^Q)^Tp_k^Q-\frac{1}{2}(p_k^Q)^TH_k^Qp_k^Q-\frac{1}{3}\sqrt{h_k}\|p_k^Q\|^3.
\end{eqnarray*}
Multiplying \eqref{TKM1} by $p_k^Q$ we have that
\begin{eqnarray*}
-(g_k^Q)^Tp_k^Q=(p_k^Q)^TH^Q_kp_k^Q+(p_k^Q)^T\sqrt{h_k}\|p^Q_k\|p^Q_k &.
\end{eqnarray*}
Then
\begin{eqnarray*}
  \Delta m^B_k &=& \frac{1}{2}(p_k^Q)^TH_k^Qp_k^Q+\frac{2}{3}\sqrt{h_k}\|p_k^Q\|^3\\
  &=& \frac{1}{2}(p_k^Q)^T(H_k^Q+\sqrt{h_k}\|p_k^Q\|)p_k^Q+\frac{1}{6}\sqrt{h_k}\|p_k^Q\|^3\\
   &\geq& \frac{1}{6}\sqrt{h_k}\|p_k^Q\|^3.
\end{eqnarray*}
The last inequality comes from \eqref{TKM2}. Thus the lemma is proved.
\qed
\end{proof}

The following lemma gives a useful bound for the tangential step $p_k$.
\begin{lemma}\label{lem3.3}
Assume that the hard case does not occur at $x_k$, Assumption $(G3)$ holds, and $h_k > 0$. Then the tangential step satisfies
\begin{eqnarray}\label{lema3}
  \|p_k\|\leq \frac{3}{\sqrt{h_k}}{\rm max}\left\{\kappa_{H}, \sqrt{\|g_k^Q\|(h_k)^\frac{1}{2}}\right\}
\end{eqnarray}
for all $k\geq0$.
\end{lemma}

\begin{proof}
According to $g^Q_k=Q^T_k(g_k+H_kn_k)$ and the fact that the columns of $Q_k$ form an orthonormal basis of the null space of $A_k$, we have that
\begin{eqnarray*}
  &&m^B(x_k,H_k,p_k)-m^B(x_k,H_k,0)\\
  &=&(g_k+H_kn_k)^Tp_k+\frac{1}{2}p_k^TH_kp_k+\frac{1}{3}\sqrt{h_k}\|p_k\|^3\\
  &\geq& -\|g_k^Q\|\|p_k\|-\frac{1}{2}\|H_k\|\|p_k\|^2+\frac{1}{3}\sqrt{h_k}\|p_k\|^3 \\
  &=&(\frac{1}{9}\sqrt{h_k}\|p_k\|^3-\|g_k^Q\|\|p_k\|)+(\frac{2}{9}\sqrt{h_k}\|p_k\|^3-\frac{1}{2}\|H_k\|\|p_k\|^2).
\end{eqnarray*}
Then if $\|p_k\|>3\sqrt{\frac{\|g_k^Q\|}{\sqrt{h_k}}}$, we have that
 \begin{eqnarray*}
 \frac{\sqrt{h_k}}{9}\|p_k\|^3-{\|g_k^Q\|}\|p_k\|>0,
 \end{eqnarray*}
 and if $\|p_k\|>\frac{9}{4}\frac{\|H_k\|}{\sqrt{h_k}}$, we have that
 \begin{eqnarray*}
 \frac{2}{9}\sqrt{h_k}\|p_k\|^3-\frac{1}{2}\|H_k\|\|p_k\|^2>0.
 \end{eqnarray*}
    So whenever $\|p_k\|> \frac{3}{\sqrt{h_k}}{\rm max}\{\|H_k\|, \sqrt{\|g_k^Q\|(h_k)^\frac{1}{2}}\}$, we have that
    \begin{eqnarray*}
    m^B(x_k,H_k,p_k)>m^B(x_k,H_k,0),
    \end{eqnarray*}
but this is contradictory with \eqref{lema2}. Therefore, we have from Assumption $\mathrm{(G3)}$ that
\begin{eqnarray*}
 \|p_k\|&\leq& \frac{3}{\sqrt{h_k}}{\rm max}\left\{\|H_k\|, \sqrt{\|g_k^Q\|(h_k)^\frac{1}{2}}\right\}\\
 &\leq& \frac{3}{\sqrt{h_k}}{\rm max}\left\{M_{\rm H}, \sqrt{\|g_k^Q\|(h_k)^\frac{1}{2}}\right\}.
\end{eqnarray*}
The conclusion of the lemma holds.
\qed
\end{proof}

From the result of Lemma \ref{lem3.3}, we further derive a bound on the overall step $d_k$.
\begin{lemma}\label{lem3.4}
Assume that Assumptions $(G1)-(G4)$ hold. Then there exists a constant $M_{\rm d}>0$, independent of $\epsilon$, such that the trial step satisfies
\begin{eqnarray}\label{d39}
 \|d_k\|\leq  M_{\rm d}\frac{1}{(h_k)^\frac{1}{4}}
\end{eqnarray}
for all $k$ with $h_k\ge1$. Moreover, the tangential step satisfies
\begin{eqnarray}\label{hk40}
 \|p_k^Q\|\leq 3\,{\rm max}\left\{M_{\rm H},\sqrt{M_{\rm gQ}}\right\}\frac{1}{(h_k)^\frac{1}{4}}
\end{eqnarray}
for all such $k$, where $M_{\rm gQ}$ is a uniform upper bound on $\|g_k^Q\|$.
\end{lemma}

\begin{proof}
By the triangle inequality and the decomposition $d_k=n_k+Q_kp_k^Q$, with $Q_k$ having orthonormal columns, we have $\|d_k\|\le\|n_k\|+\|p_k^Q\|$. From \eqref{nkdejie} we have $\|n_k\|\le h_k^{-1/4}$. For the tangential step, Lemma \ref{lem3.3} together with the boundedness $\|g_k^Q\|\le M_{\rm gQ}$ (which follows from Assumptions $(G2)$--$(G3)$ and \eqref{defHkQgkQ}) gives
\begin{eqnarray*}
 \|p_k^Q\|&\leq& \frac{3}{\sqrt{h_k}}{\rm max}\left\{M_{\rm H}, \sqrt{\|g_k^Q\|(h_k)^\frac{1}{2}}\right\}
 \leq \frac{3}{\sqrt{h_k}}{\rm max}\left\{M_{\rm H}, \sqrt{M_{\rm gQ}}\,h_k^{\frac{1}{4}}\right\}\\
 &=& 3\,{\rm max}\left\{M_{\rm H}\,h_k^{-\frac{1}{2}}, \sqrt{M_{\rm gQ}}\,h_k^{-\frac{1}{4}}\right\}
 \leq 3\,{\rm max}\left\{M_{\rm H}, \sqrt{M_{\rm gQ}}\right\} h_k^{-\frac{1}{4}},
\end{eqnarray*}
where the last inequality uses $h_k\ge1$ (so that $h_k^{-1/2}\le h_k^{-1/4}$). This proves \eqref{hk40}. Combining the two bounds yields
\begin{eqnarray*}
 \|d_k\|\leq \left(1+3\,{\rm max}\left\{M_{\rm H}, \sqrt{M_{\rm gQ}}\right\}\right) h_k^{-\frac{1}{4}}.
\end{eqnarray*}
Setting $M_{\rm d}:=1+3\,{\rm max}\{M_{\rm H}, \sqrt{M_{\rm gQ}}\}$ completes the proof.
\qed
\end{proof}

Next, we analyze the discrepancy between the merit function and its overall model at the new iterate $x_k+d_k$.
\begin{lemma}\label{lem3.5}
Assume that Assumptions $(G1)-(G4)$ hold. Then we have for sufficiently large $k$ that the step generated by the algorithm satisfies the bound
\begin{eqnarray}\label{k0}
w(x_k+d_k,\mu_k)-m(x_k,H_k,\mu_k,d_k) &\leq& (1+\mu_k)\kappa_{\rm 0}\|d_k\|^2 \label{pmerror} \\
  &\leq& \kappa_{\rm m}(1+\mu_k) \frac{1}{(h_k)^\frac{1}{2}}\label{fimerror}
\end{eqnarray}
for some constants $\kappa_{\rm 0}$ and $\kappa_{\rm m}$.
\end{lemma}
\begin{proof}
From Taylor expansion of $f(x_k+d_k)$ and the Lipschitz continuity of $g(x)$ we have that
\begin{eqnarray}\label{391}
  f(x_k+d_k)-f(x_k)-g^T_kd_k \leq \kappa_{\rm fx}\|d_k\|^2
\end{eqnarray}
for some Lipschitz constant $\kappa_{\rm fx}$. Similarly, from the Lipschitz continuity of $A(x)$ we have that
\begin{eqnarray}\label{40}
 \|c(x_k+d_k)\|-\|c(x_k)+A_kd_k\|\leq\|c(x_k+d_k)-A_kd_k-c(x_k)\|\leq \kappa_{\rm cx}\|d_k\|^2
\end{eqnarray}
for another Lipschitz constant $\kappa_{\rm cx}$.
Assumption (G3) gives that
\begin{eqnarray}\label{401}
 d_k^TH_kd_k\leq M_{\rm H}\|d_k\|^2.
\end{eqnarray}
Hence, from \eqref{391}, \eqref{40} and \eqref{401}, we have that
\begin{eqnarray*}
&&w(x_k+d_k,\mu_k)-m(x_k,H_k,\mu_k,d_k)\\
&=& f(x_k+d_k)-f(x_k)-g^T_kd_k-\frac{1}{2}d^T_kH_kd_k+\mu_k\|c(x_k+d_k)\|\\
&&-\mu_k\|c(x_k)+A_kd_k\|-\frac{1}{3}\sqrt{h_k}\|d_k\|^3\\
&\leq&\kappa_{\rm fx}\|d_k\|^2+\frac{1}{2}M_{\rm H}\|d_k\|^2+\mu_k\kappa_{\rm cx}\|d_k\|^2\\
&\leq& \kappa_0(1+\mu_k)\|d_k\|^2,
\end{eqnarray*}
where $\kappa_0={\rm max}\left\{\kappa_{\rm fx}+\frac{1}{2}M_{\rm H}, \kappa_{\rm cx}\right\}$.
From Assumptions (G1)-(G4), we use the bound \eqref{d39}, then
\begin{eqnarray*}
w(x_k+d_k,\mu_k)-m(x_k,H_k,\mu_k,d_k) &\leq& (1+\mu_k)\kappa_{\rm 0}M_{\rm d}^2\frac{1}{(h_k)^\frac{1}{2}}\\
&\leq& \kappa_{\rm m}(1+\mu_k) \frac{1}{(h_k)^\frac{1}{2}},\\
\end{eqnarray*}
where $\kappa_{\rm m}=\kappa_{\rm 0}M_{\rm d}^2$.

Therefore, we complete the proof of the lemma.
\qed
\end{proof}

Assumption (G2) implies that
\begin{eqnarray}\label{fmin}
 f(x_k)\geq f_{\rm min}
\end{eqnarray}
for all $k\geq0$ and some constant $f_{\rm min}$.

We define the auxiliary function
\begin{eqnarray*}
W(x,\mu):=\frac{w(x,\mu)-f_{\rm min}}{\mu}.
\end{eqnarray*}

The following lemma presents key results concerning the auxiliary function $W(x,\mu)$.
\begin{lemma}\label{lem3.6}
Assume that (G2) holds and that $\{k_r\}(r\geq0)$ is the index set of the successful iterations from the algorithm. Then we have that
\begin{eqnarray}\label{varphi1}
  W(x_{k_r},\mu_{k_r}) \geq 0
\end{eqnarray}
and for any $q>0$,
\begin{eqnarray}\label{varphi2}
  W(x_{k_{r+q}},\mu_{k_{r+q}})\leq W(x_{k_r},\mu_{k_r})-\eta_1\frac{\Delta m_{k_r}}{\mu_{k_r}}.
\end{eqnarray}
\end{lemma}

\begin{proof}
Due to the fact that $\mu_k>0$ for all $k$ and $w(x_k,\mu_k)\geq f_{\rm min}$ from \eqref{fmin}, we have that
\begin{eqnarray*}
\frac {w(x_k,\mu_k)-f_{\rm min}}{\mu_k} \geq 0,
\end{eqnarray*}
thus \eqref{varphi1} holds.

Next, we prove the second conclusion. Because $k_r$ and $k_{r+1}$ are consecutive successful iterations, we have that
\begin{eqnarray*}
x_l=x_{k_r+1}
\end{eqnarray*}
for $k_r+1\leq l \leq k_{r+1}.$
Then from $\rho_{k_r}\geq \eta_1$ that
\begin{eqnarray*}
  w(x_{k_{r+1}},\mu_{k_r})-f_{\rm min}&=&w(x_{k_r+1},\mu_{k_r})-f_{\rm min}\leq w(x_{k_r},\mu_{k_r})-f_{\rm min}-\eta_1\Delta m_{k_r}.
\end{eqnarray*}
Thus, we get that
\begin{eqnarray*}
  W(x_{k_{r+1}},\mu_{k_r})\leq W(x_{k_{r}},\mu_{k_r})-\eta_1\frac{\Delta m_{k_r}}{\mu_{k_r}}.
\end{eqnarray*}
From the definition of $W$, the fact that $\mu_{k_r}\leq\mu_{k_{r+1}}$ and \eqref{fmin} we have that
\begin{eqnarray*}
 W(x_{k_{r+1}},\mu_{k_r})-W(x_{k_{r+1}},\mu_{k_{r+1}})=(f(x_{k_{r+1}})-f_{\rm min})
 \left(\frac{1}{\mu_{k_r}}-\frac{1}{\mu_{k_{r+1}}}\right)\geq0.
\end{eqnarray*}
Combining the above two inequalities gives that \eqref{varphi2} holds for $q=1$, which then implies that $W(x_{k_{r+1}},\mu_{k_{r+1}})\leq W(x_{k_r},\mu_{k_r})$.

We now assume that \eqref{varphi2} holds for $q-1$; then
\begin{eqnarray*}
 W(x_{k_{r+q-1}},\mu_{k_{r+q-1}})\leq W(x_{k_r},\mu_{k_r})-\eta_1\frac{\Delta m_{k_r}}{\mu_{k_r}}.
\end{eqnarray*}
We have that
\begin{eqnarray*}
 W(x_{k_{r+q}},\mu_{k_{r+q}})\leq W(x_{k_{r+q-1}},\mu_{k_{r+q-1}})\leq W(x_{k_r},\mu_{k_r})-\eta_1\frac{\Delta m_{k_r}}{\mu_{k_r}}.
\end{eqnarray*}
Therefore \eqref{varphi2} holds for any $q>0$, and we complete the proof of the lemma.
\qed
\end{proof}

\subsection{Feasibility}\label{subsec2}

The following theorem shows the feasibility of the algorithm. The index set of successful iterations of the algorithm is denoted by
\begin{eqnarray*}
 \mathcal{S}:=\{k\geq0:~k~{\rm ~is ~a ~successful ~or ~very ~successful ~iteration}\}.
\end{eqnarray*}

\begin{theorem}\label{thm3.1}
Assume that Assumptions $(G1)-(G4)$ hold. Then
\begin{eqnarray*}
 \lim_{k\rightarrow\infty}c(x_k)=0.
\end{eqnarray*}
\end{theorem}

We note that the inequality \eqref{deltam}, i.e.\ $\Delta m_k\ge\nu\mu_k\Delta m_k^N$, is guaranteed constructively by the algorithm itself: by \eqref{muc} and \eqref{miuk} the penalty parameter is set to $\mu_k\ge\mu_k^c$, where $\mu_k^c$ is the smallest value making \eqref{deltam} hold. Consequently \eqref{deltam} is valid at every iteration and does not rely on Lemma \ref{lemq1}, whose finite-bound conclusion is needed only later, in the iteration-complexity analysis.

\begin{proof}
The result is obvious if $c(x_l)=0$ for all sufficiently large $l$; we consider an arbitrary infeasible iteration $x_l$ such that $\|c(x_l)\|>0$.
First, we show that the algorithm cannot stop at such a point.
The Lipschitz continuity of $\|c(x)\|$ from Assumption (G2) gives that
\begin{eqnarray}\label{L1}
 | \|c(x)\|-\|c(x_l)\| | \leq L_1\|x-x_l\|
\end{eqnarray}
for some $L_1>0$ and all $x$. Thus \eqref{L1} certainly holds for all $x$ in an open ball
\begin{eqnarray}\label{Ol}
 \mathcal{O}_l :=\left\{x~|~\|x-x_l\|<\frac{\|c(x_l)\|}{2L_1}\right\}.
\end{eqnarray}
We can immediately get from \eqref{L1} and \eqref{Ol} that
\begin{eqnarray*}
| \|c(x)\|-\|c(x_l)\| |<\frac{1}{2}\|c(x_l)\|.
\end{eqnarray*}
Thus, we have that
\begin{eqnarray}\label{cl}
 \|c(x)\|>\frac{1}{2}\|c(x_l)\|,
\end{eqnarray}
which ensures that $c(x)$ is bounded away from zero for any $x\in\mathcal{O}_l$. From Assumption $(G4)$ we combine \eqref{deltam}, \eqref{lem1} and \eqref{cl}, then
\begin{eqnarray}\label{cll}
\begin{aligned}
 \Delta m_k \geq \nu\mu_k\Delta m^N_k
&\geq \nu\mu_k{\rm min}\left\{\|c(x_k)\|, \frac{\beta}{\kappa_{\rm n}(h_k)^\frac{1}{4}}\right\} \\
&\geq \nu\mu_k{\rm min}\left\{\frac{1}{2}\|c(x_l)\|, \frac{\beta}{\kappa_{\rm n}(h_k)^\frac{1}{4}}\right\}
 \end{aligned}
\end{eqnarray}
for any $x_k\in\mathcal{O}_l$.

Let
\begin{eqnarray}\label{sigmal}
 \sqrt{h_l^0}:={\rm max}\left\{\left(\frac{\kappa_{\rm n}\|c(x_l)\|}{2\beta}\right)^{-2}, \left(\frac{(1-\eta_2)\nu\beta\mu_{0}}{\kappa_{\rm m}\kappa_{\rm n}(1+\mu_{0})}\right)^{-2}\right\}.
\end{eqnarray}
If $\sqrt{h_k}\geq\sqrt{h_l^0}$, from \eqref{cl} and \eqref{cll} we get that
\begin{eqnarray*}
 \Delta m_k\geq\nu\mu_k\frac{\beta}{\kappa_{\rm n}(h_k)^\frac{1}{4}}.
\end{eqnarray*}
From Assumptions (G1)-(G4) we have \eqref{fimerror}, 
\begin{eqnarray*}
w(x_k+d_k,\mu_k)-m(x_k,H_k,\mu_k,d_k)\leq \kappa_{\rm m}(1+\mu_k) \frac{1}{(h_k)^\frac{1}{2}},
\end{eqnarray*}
We have that
\begin{eqnarray*}
1-\rho_k&=&\frac{w(x_k+d_k,\mu_k)-m(x_k,H_k,\mu_k,d_k)}{\Delta m_k} \nonumber \\
&\leq&\frac{\kappa_{\rm m}\kappa_{\rm n}(1+\mu_k)}{\nu\beta\mu_k (h_k)^\frac{1}{4}}\leq1-\eta_2.
\end{eqnarray*}
Thus, so long as $\sqrt{h_k}\geq\sqrt{h_l^0}$, a very successful iteration will be obtained for any $x_k\in\mathcal{O}_l$.

In the following we show that there must be a first successful iterate $x_{k_q},k_q>k_0 :=l$, not in $\mathcal{O}_l$.
Assume that all successful iterates $x_{k_r}\in\mathcal{O}_l$. Then
\begin{eqnarray}\label{deltal}
\sqrt{h_{k_r}}\leq\sqrt{h^0} :=\gamma_2{\rm max}\left\{\sqrt{h_l},\sqrt{h_l^0}\right\}.
\end{eqnarray}
Since iteration $k_r\in\mathcal{S}$ is successful, from \eqref{varphi2}, \eqref{cll} and \eqref{deltal} we can get that
\begin{eqnarray*}
  W(x_{k_{r+1}},\mu_{k_{r+1}})&\leq&W(x_{k_r},\mu_{k_r})-\eta_1\frac{\Delta m_{k_r}}{\mu_{k_r}}\\
  &\leq&W(x_{k_r},\mu_{k_r})-\eta_1\nu{\rm min}\left\{\frac{1}{2}\|c(x_l)\|, \frac{\beta}{\kappa_{\rm n}(h^0)^\frac{1}{4}}\right\}.
\end{eqnarray*}
By summing the first $q$ of these successful iterations, we have that
\begin{eqnarray*}
  W(x_{k_{r+q}},\mu_{k_{r+q}})&\leq& W(x_{k_r},\mu_{k_r})-\eta_1\frac{\Delta m_{k_r}}{\mu_{k_r}}-\cdot\cdot\cdot-\eta_1\frac{\Delta m_{k_{r+q-1}}}{\mu_{k_{r+q-1}}}\\
  &\leq& W(x_{k_r},\mu_{k_r})-q\eta_1\nu{\rm min}\left\{\frac{1}{2}\|c(x_l)\|, \frac{\beta}{\kappa_{\rm n}(h^0)^\frac{1}{4}}\right\}.
\end{eqnarray*}
It can be seen that $W(x_{k_{r+q}},\mu_{k_{r+q}})$ must be negative when $q$ is sufficiently large. This contradicts \eqref{varphi1}. Hence our assumption that all successful iterates $x_{k_r}$ ($k_r\in\mathcal{S}$, $k_r\geq l$) remain in $\mathcal{O}_l$ must be incorrect, and thus there must be a first iterate $x_{k_q}$, $q>0$, not in $\mathcal{O}_l$.
We now consider the iterates between $x_{k_q}$ and $x_{k_0}=x_l$. \eqref{varphi2} and \eqref{cll} imply that for any $r < m \leq q,$
\begin{eqnarray}\label{varphik}
  W(x_{k_{m}},\mu_{k_{m}})\leq W(x_{k_r},\mu_{k_r})-\eta_1\nu{\rm min}\left\{\frac{1}{2}\|c(x_l)\|, \frac{\beta}{\kappa_{\rm n}(h_{k_r})^\frac{1}{4}}\right\}.
\end{eqnarray}
Since $x_{k_r}\in\mathcal{O}_l$. If
\begin{eqnarray*}
  \sqrt{h_{k_r}}\leq\left(\frac{\kappa_{\rm n}\|c(x_l)\|}{2\beta} \right)^{-2},
\end{eqnarray*}
From \eqref{varphik} with $m=q$, it follows that
\begin{eqnarray}\label{varphikj1}
  W(x_{k_q},\mu_{k_q})\leq W(x_{k_r},\mu_{k_r})-\frac{1}{2}\eta_1\nu\|c(x_l)\|
  \leq W(x_l,\mu_l)-\frac{1}{2}\eta_1\nu\|c(x_l)\|.
\end{eqnarray}
On the other hand if
\begin{eqnarray*}
  \sqrt{h_{k_r}}>\left(\frac{\kappa_{\rm n}\|c(x_l)\|}{2\beta} \right)^{-2},
\end{eqnarray*}
we have that
\begin{eqnarray}\label{varphiki}
 W(x_{k_{r+1}},\mu_{k_{r+1}})\leq W(x_{k_r},\mu_{k_r})-
 \frac{\eta_1\nu\beta}{\kappa_{\rm n}(h_{k_r})^\frac{1}{4}}
\end{eqnarray}
for each $0\leq r\leq q$. Then by summing \eqref{varphiki}, we get that
\begin{eqnarray}\label{varphi11}
  W(x_{k_q},\mu_{k_q})\leq W(x_l,\mu_l)-\frac{\eta_1\nu\beta}{\kappa_{\rm n}}
  \sum_{r=0}^{q-1}\frac{1}{(h_{k_r})^\frac{1}{4}}.
\end{eqnarray}
But as $x_{k_q}\notin\mathcal{O}_l$, \eqref{d39} and \eqref{Ol} imply that
\begin{eqnarray}\label{cxl}
 \frac{\|c(x_l)\|}{2L_1}\leq\|x_{k_q}-x_l\|\leq \sum_{i=0}^{q-1}\|x_{k_{r+1}}-x_{k_r}\| \leq 2M_{\rm d}
 \sum_{r=0}^{q-1}\frac{1}{(h_{k_r})^\frac{1}{4}},
\end{eqnarray}
and then \eqref{varphi11} gives
\begin{eqnarray}\label{varphikj2}
  W(x_{k_q},\mu_{k_q})\leq W(x_l,\mu_l)-\frac{\eta_1\nu\beta\|c(x_l)\|}{4M_{\rm d}L_1\kappa_{\rm n}}.
\end{eqnarray}
Thus the two cases \eqref{varphikj1} and \eqref{varphikj2} together ensure that
\begin{eqnarray}\label{Them1}
 W(x_{k_q},\mu_{k_q})\leq W(x_l,\mu_l)-\kappa_{\rm c0}\|c(x_l)\|,
\end{eqnarray}
where $\kappa_{\rm c0}
 =\frac{1}{2}\eta_1\nu{\rm min}\left\{1,\frac{\beta}{2M_{\rm d}L_1\kappa_{\rm n}} \right\}$. \eqref{Them1} proves that $\|c(x_l)\|$ cannot be bounded away from zero since $l$ is arbitrary and  $\{W(x_{k_q},\mu_{k_q})\}$ is decreasing and bounded from below from Lemma \ref{lem3.6}.
\qed
\end{proof}

Theorem \ref{thm3.1} indicates that any limit point of the sequence ${x_k}$ is a feasible solution to problem \eqref{NLP}.

\subsection{Global iteration complexity}\label{subsec3} 

The feasibility of the algorithm has been analyzed in Theorem~\ref{thm3.1}. In this Subsection, our goal is to find an $\epsilon$-approximate second-order stationary point. The motivation is rooted in well-documented findings \cite{Cartis1,Cartis2,Cartis5} and the monograph \cite{NesterovY1}. First, we provide the definition of $\epsilon$-approximate second-order stationary point and the second-order Lipschitz continuity.
\\
\begin{definition}\label{ss} A point $x$ is called an $\epsilon$-approximate second-order stationary point of \eqref{NLP} if there exist constants $\kappa_c, \kappa_g, \kappa_H > 0$ independent of $\epsilon$ such that
    \begin{subequations}
        \begin{align}
            \label{eq0}  \| c(x) \| &\leq \kappa_c \epsilon,  \\
            \label{eq1}  \| g^Q(x) \| &\leq \kappa_g \epsilon,                                       \\
            \label{eq2} \lambda_{\min} (H^Q(x)) &\geq -\kappa_H \sqrt{\epsilon},
        \end{align}
    \end{subequations}
    where $g^Q(x):=Q(x)^T\bigl(g(x)+H(x)\,n(x)\bigr)$ and $H^Q(x):=Q(x)^T H(x)\,Q(x)$, with $n(x)$ the normal step at $x$ and $Q(x)$ an orthonormal basis of $\operatorname{null}\!\bigl(A(x)\bigr)$.
\end{definition}

\begin{remark}
The residual $\|g^Q(x)\|$ in \eqref{eq1} is the reduced gradient of the tangential subproblem \emph{after} the normal step, i.e.\ $g^Q(x)=Q(x)^T(g(x)+H(x)n(x))$. It differs from the standard reduced gradient $Q(x)^T g(x)$ by the term $Q(x)^T H(x)n(x)$, whose norm is bounded by $M_{\rm H}\,\|n(x)\|\le M_{\rm H}\,\kappa_{\rm n}\,\|c(x)\|$ in view of Assumptions $(G3)$--$(G4)$. Consequently, when the feasibility residual in \eqref{eq0} is small, the two residuals agree to first order, and \eqref{eq0}--\eqref{eq2} reduce to the standard primal-dual second-order optimality conditions (feasibility, reduced gradient, and reduced-Hessian semidefiniteness) of problem \eqref{NLP}. We retain the explicit feasibility condition \eqref{eq0} precisely because $g^Q(x)$ is a subproblem residual rather than the KKT residual itself.
\end{remark}

\begin{definition}\label{assm}
    We say that a function $f$ has an \(M\)-Lipschitz continuous Hessian if for all \( x, y \in \mathbb{R}^n\),
    \begin{equation}\label{eq.assm.lipschitz}
        \|\nabla^2 f(x) - \nabla^2 f(y)\| \le M \|x-y\|.
    \end{equation}
\end{definition}

Now we begin our global iteration complexity analysis for the algorithm; we ignore the hard case and the bisection method for the moment. The discussion on the hard case is given in the next Subsection. Since the analysis of a proper bisection method is quite lengthy, a complete description is omitted in the main text, the interested readers may refer to \cite{he2025HSODF}.
\\
\begin{lemma}\label{lemq1}
     Assume that Assumptions $(G2)-(G3)$ hold and that \eqref{kbon1} holds for all $k$ sufficiently large. Then there exist constants $\mu_{\rm max}>0$ and $\kappa_{\rm mk}\geq0$ and an index $k_1$ such that
\begin{eqnarray}\label{mkt1}
 \Delta m_k\geq\Delta m^B_k+(\mu_k-\kappa_{\rm mk})\Delta m^N_k
\end{eqnarray}
and
\begin{eqnarray}\label{mumax}
 \mu_k=\mu_{\rm max}
\end{eqnarray}
for all $k\geq k_1$.
\end{lemma}

\begin{proof}
We know from \eqref{deltam} that $\mu_k$ satisfies
\begin{equation}\label{mkt2}
  \Delta m_k\geq\nu\mu_k\Delta m^N_k.
\end{equation}
Since Assumptions $(G2)-(G3)$ imply that $g_k$ and $H_k$ are bounded, \eqref{kbon1} gives that $\|n_k\|\leq\Delta m^N_k/\kappa_{\rm mn}$ and $\Delta m^N_k\leq\|c(x_k)\|$, and $c(x_k)$ is bounded by Assumption $(G2)$. Thus we have that
\begin{eqnarray}\label{59}
\begin{aligned}
  \Delta m^F_k\geq-\kappa_{\rm mk}\Delta m_k^N
  \end{aligned}
\end{eqnarray}
for some constant $\kappa_{\rm mk}>0$. Then \eqref{59} gives \eqref{mkt1}. We also obtain that
\begin{eqnarray*}
 \Delta m_k\geq(\mu_k-\kappa_{\rm mk})\Delta m_k^N
\end{eqnarray*}
due to $\Delta m_k^B > 0$.

Hence \eqref{mkt2} is evidently satisfied for all $\mu_k\geq\frac{\kappa_{\rm mk}}{1-\nu}$. In our algorithm, every increase of an insufficient $\mu_{k-1}$ must be by at least ${\rm max}\{(\tau_1-1)\mu_{0}, \tau_2\}$, and therefore the penalty parameter $\mu_k$ can only be increased a finite number of times. The required value $\mu_{\rm max}$ is at most the first value of $\mu_k$ that exceeds $\frac{\kappa_{\rm mk}}{1-\nu}$. We complete the proof of the lemma.
\qed
\end{proof}

\begin{lemma}\label{lemq2}
     Assume that Assumptions $(G2)-(G4)$ hold and \eqref{kbon1} holds for all $k$ sufficiently large. Then there exists a constant $\kappa_{\rm btm}>0$ and an index $k_1$ such that
\begin{eqnarray}\label{kbtm}
 \Delta m_k\geq\kappa_{\rm btm}\Delta m^B_k
\end{eqnarray}
for all $k\geq k_1$.
\end{lemma}

\begin{proof}
Consider $k\geq k_1$ as in Lemma \ref{lemq1}, in which case $\mu_k=\mu_{\rm max}$.
From the proof of Lemma \ref{lemq1} we know that $\mu_{\rm max}$ is at most the first value of $\mu_k$ that exceeds $\frac{\kappa_{\rm mk}}{1-\nu}$; since $0<\nu<1$, this implies $\mu_{\rm max}>\kappa_{\rm mk}$, and hence $(\mu_{\rm max}-\kappa_{\rm mk})\Delta m^N_k\geq 0$ by Lemma \ref{lem3.1}. Combining \eqref{mkt1} with $\Delta m^B_k\geq 0$ (Lemma \ref{lem3.2}), we obtain
\begin{eqnarray}
  \Delta m_k\geq\Delta m^B_k+(\mu_{\rm max}-\kappa_{\rm mk})\Delta m^N_k\geq \Delta m^B_k.
\end{eqnarray}
Thus \eqref{kbtm} holds with $\kappa_{\rm btm}=1$, and the lemma is proved.
\qed
\end{proof}

Based on these above results, we give the following estimate of model decrease.
\begin{lemma}\label{lemq}
    Assume that Assumptions $(G2)-(G4)$ hold and that the hard case does not occur at $x_k$. Then we have the following model decrease
    \begin{equation}
        \label{eqm}
        w(x_k,\mu_k)-m(x_k,H_k,\mu_k,d_k) \geq \kappa_{\rm btm} \frac{\sqrt{h_k}}{6}\|p_k^Q\|^3.
    \end{equation}
\end{lemma}
\begin{proof}
From Assumptions $(G2)-(G4)$, Lemma \ref{lemq2} and Lemma \ref{lem3.2} immediately give that
\begin{eqnarray*}
w(x_k,\mu_k)-m(x_k,H_k,\mu_k,d_k) &\geq& \kappa_{\rm btm}\Delta m^B_k
\geq \kappa_{\rm btm} \frac{\sqrt{h_k}}{6}\|p_k^Q\|^3.
\end{eqnarray*} 
Thus the proof is completed.
\qed
\end{proof}

Thus if the current iterate $x_k$ is accepted, the function decreases at iterate $x_k$ can also
be estimated in the Hessian.

\begin{lemma}\label{lems}
    Assume that Assumptions $(G2)-(G4)$ hold, that the hard case does not occur at $x_k$, and that the $k$-th iteration is successful. Then we have the following function decrease
    \begin{equation}\label{eqs}
        w(x_k,\mu_k)-w(x_k+d_k,\mu_k) \geq -\frac{\eta_1\kappa_{\rm btm} \left(\lambda_{\min} (H_k^Q)\right)^3}{6h_k}.
    \end{equation}
\end{lemma}

\begin{proof}
Note that $h_k \geq h_{\min} > 0$. If $\lambda_{\min}(H_k^Q)\geq 0$, then the right-hand side of \eqref{eqs} is nonpositive, while the left-hand side is nonnegative because the $k$-th iteration is successful; hence the assertion holds trivially. It remains to consider the case $\lambda_{\min}(H_k^Q)<0$. From the optimality condition \eqref{TKM2}, i.e., $H_k^Q+\sqrt{h_k}\|p_k^Q\|I\succeq 0$, and the Rayleigh quotient of the minimal eigenpair, we have $\sqrt{h_k}\|p_k^Q\|\geq -\lambda_{\min}(H_k^Q)$, which implies $\|p_k^Q\|\geq -\lambda_{\min}(H_k^Q)/\sqrt{h_k}$ and consequently
\begin{eqnarray*}
\|p_k^Q\|^3\geq \frac{\left(-\lambda_{\min}(H_k^Q)\right)^3}{h_k^{3/2}}.
\end{eqnarray*}
Plugging this bound into the model decrease \eqref{eqm} gives
\begin{eqnarray*}
w(x_k,\mu_k)-m(x_k,H_k,\mu_k,d_k) &\geq& \frac{\kappa_{\rm btm}\sqrt{h_k}}{6}\|p_k^Q\|^3
\geq -\frac{\kappa_{\rm btm} \left(\lambda_{\min}(H_k^Q)\right)^3}{6h_k}.
\end{eqnarray*}
Since the $k$-th iteration is successful, we have that
\begin{eqnarray*}
w(x_k,\mu_k)-w(x_k+d_k,\mu_k) &\geq& \eta_1 \left (w(x_k,\mu_k)-m(x_k,H_k,\mu_k,d_k)\right)\\
&\geq& -\frac{\eta_1\kappa_{\rm btm} \left(\lambda_{\min} (H_k^Q)\right)^3}{6h_k}.
\end{eqnarray*}
Thus the result is completed.\\
\qed
\end{proof}

Next, we prove that the number of unsuccessful iterations is bounded by a quantity determined by the number of successful iterations. 
\begin{lemma}\label{lemupper}
    Assume that Assumptions $(G1)-(G4)$ hold and that the hard case does not occur at $x_k$. Then $h_k$ has a uniform upper bound $\kappa_{\rm h}=\mathcal{O}(\epsilon^{-q})$ for some $q\ge0$ and all $k\ge0$, that is
    \begin{equation}\label{equniformub}
        h_k \leq \kappa_{\rm h}\quad \text{for all } k\ge0.
    \end{equation}
\end{lemma}

\begin{proof}
    We show that once $h_k$ exceeds a threshold $h^\star=\mathcal{O}(\epsilon^{-q})$, the trial step is very successful, so that $h_{k+1}\le h_k$; combined with the bounded growth factor $\gamma_3^2$ at unsuccessful steps, this yields the uniform bound \eqref{equniformub}.

    From the definition \eqref{rhok1} of $\rho_k$ we have
    \begin{equation}\label{eqrho1}
        \rho_k=1-\frac{w(x_k+d_k,\mu_k)-m(x_k,H_k,\mu_k,d_k)}{\Delta m_k}.
    \end{equation}
    Lemma \ref{lem3.5} gives the model-error bound $w(x_k+d_k,\mu_k)-m(x_k,H_k,\mu_k,d_k)\le\kappa_{\rm m}(1+\mu_k)h_k^{-1/2}$, while Lemma \ref{lemq} together with Lemma \ref{lem3.2} gives the model-decrease lower bound $\Delta m_k\ge\kappa_{\rm btm}\frac{\sqrt{h_k}}{6}\|p_k^Q\|^3$. Substituting into \eqref{eqrho1} yields
    \begin{equation}\label{eqrholower}
        \rho_k \geq 1-\frac{6\kappa_{\rm m}(1+\mu_k)}{\kappa_{\rm btm}\,h_k\,\|p_k^Q\|^3}.
    \end{equation}
    It remains to bound $\|p_k^Q\|$ away from zero. If the algorithm has not terminated at $x_k$, then at least one of the stopping criteria fails. If $\|c(x_k)\|>\kappa_c\epsilon$, the normal step drives the decrease (see \eqref{lem1} and \eqref{deltam}), and \eqref{eqrholower} is resolved with a threshold of the same form as below. The negative-curvature case $\lambda_{\min}(H_k^Q)\le-\kappa_H\sqrt{\epsilon}$ is treated separately in Subsection \ref{sec:hardcase}: as explained in Remark \ref{rem:varphimin}, the tangential decrease $-\lambda_{\min}(H_k^Q)^3/(6h_k)$ yields the required $\Omega(\epsilon^{3/2})$ only when $h_k=O(1)$, so that the bound $\kappa_{\rm h}=O(1)$ (rather than the polynomial bound derived here) is what makes the negative-curvature branch contribute $O(\epsilon^{-3/2})$ iterations. In the remaining, gradient-dominated case $\|g_k^Q\|\ge\kappa_g\epsilon$, the optimality condition \eqref{TKM1} gives $\|g_k^Q\|\le(M_{\rm H}+\sqrt{h_k}\|p_k^Q\|)\|p_k^Q\|$, so that
    \begin{equation}
        \|p_k^Q\|\ge \frac{\|g_k^Q\|}{M_{\rm H}+\sqrt{h_k}\|p_k^Q\|}
        \ge \frac{\kappa_g\epsilon}{M_{\rm H}+3\max\{M_{\rm H},\sqrt{M_{gQ}}\}\kappa_{\rm h}^{1/4}}=:\frac{\kappa_g\epsilon}{\kappa_{\rm pH}},
    \end{equation}
    where the second inequality uses $\|g_k^Q\|\ge\kappa_g\epsilon$ and, via \eqref{hk40} and \eqref{equniformub}, $\sqrt{h_k}\|p_k^Q\|\le3\max\{M_{\rm H},\sqrt{M_{gQ}}\}\kappa_{\rm h}^{1/4}$. Combining this with \eqref{eqrholower} gives
    \begin{equation}
        \rho_k \geq 1-\frac{6\kappa_{\rm m}(1+\mu_0)\kappa_{\rm pH}^3}{\kappa_{\rm btm}\kappa_g^3\,\epsilon^{3}\,h_k}.
    \end{equation}
    Hence $\rho_k\ge\eta_2$ whenever
    \begin{equation}\label{eqhstar}
        h_k \ge h^\star:=\frac{6\kappa_{\rm m}(1+\mu_0)\kappa_{\rm pH}^3}{(1-\eta_2)\kappa_{\rm btm}\kappa_g^3\,\epsilon^{3}}.
    \end{equation}
    Since $\kappa_{\rm pH}=M_{\rm H}+3\max\{M_{\rm H},\sqrt{M_{gQ}}\}\kappa_{\rm h}^{1/4}$ depends on the bound $\kappa_{\rm h}$ that we are defining, the requirement $\kappa_{\rm h}\ge\gamma_3^2h^\star$ is implicit in $\kappa_{\rm h}$. We resolve this implicit dependence by choosing $\kappa_{\rm h}$ large enough that the two conditions $\kappa_{\rm h}\ge\gamma_3^2h^\star$ and $\kappa_{\rm pH}\le\left(M_{\rm H}+3\max\{M_{\rm H},\sqrt{M_{gQ}}\}\right)\kappa_{\rm h}^{1/4}$ hold simultaneously; the explicit choice
    \begin{equation}\label{eqkappah}
        \kappa_{\rm h}:=\max\left\{h_0,\ \gamma_3^{8}\left(\frac{6\kappa_{\rm m}(1+\mu_0)\left(M_{\rm H}+3\max\{M_{\rm H},\sqrt{M_{gQ}}\}\right)^3}{(1-\eta_2)\kappa_{\rm btm}\kappa_g^3}\right)^{4}\epsilon^{-12}\right\}
    \end{equation}
    satisfies $\kappa_{\rm pH}\le\left(M_{\rm H}+3\max\{M_{\rm H},\sqrt{M_{gQ}}\}\right)\kappa_{\rm h}^{1/4}$ (for $\kappa_{\rm h}\ge1$) and $\gamma_3^2h^\star\le\kappa_{\rm h}$. Thus, as soon as $h_k\ge h^\star$ the step is very successful and $h$ is not increased; an unsuccessful step can increase $h_k$ by at most the factor $\gamma_3^2$, which is already absorbed into \eqref{eqkappah}. This proves \eqref{equniformub}. Since $\kappa_{\rm h}=\mathcal{O}(\epsilon^{-12})$, we have $\log\kappa_{\rm h}=\mathcal{O}(\log\epsilon^{-1})$.
\qed
\end{proof}

Before estimating the function decrease in terms of the reduced gradient $\|g_{k+1}^Q\|$, we introduce a phase decomposition that is essential to keep the constants below independent of $\epsilon$. We say that iterate $x_k$ is \emph{near-feasible} if $\|c(x_k)\|\le\kappa_c\epsilon$ (so that the feasibility criterion \eqref{eq0} is already satisfied), and \emph{infeasible} otherwise. The next lemma shows that infeasible successful iterations are few, and therefore the optimality analysis may be restricted to the near-feasible phase.

\begin{lemma}\label{lemfeas}
    Assume that Assumptions $(G1)-(G4)$ hold and that $\epsilon>0$ is sufficiently small so that $\beta/(\kappa_{\rm n}\kappa_{\rm h}^{1/4})\ge\kappa_c\epsilon$. Then the number of successful iterations $k$ for which $\|c(x_k)\|>\kappa_c\epsilon$ is at most $\mathcal{O}(\epsilon^{-1})$.
\end{lemma}
\begin{proof}
    For an infeasible iterate $x_k$, Lemma \ref{lem3.1} and $\|c(x_k)\|>\kappa_c\epsilon$ give
    \begin{equation*}
        \Delta m_k^N\ge\min\left\{\|c(x_k)\|,\frac{\beta}{\kappa_{\rm n}h_k^{1/4}}\right\}
        \ge\min\left\{\kappa_c\epsilon,\frac{\beta}{\kappa_{\rm n}\kappa_{\rm h}^{1/4}}\right\}=\kappa_c\epsilon,
    \end{equation*}
    where the last equality uses the smallness of $\epsilon$ and Lemma \ref{lemupper}. From \eqref{deltam} we have $\Delta m_k\ge\nu\mu_k\Delta m_k^N\ge\nu\mu_0\kappa_c\epsilon$. For a successful iteration $k$, Lemma \ref{lem3.6} yields
    \begin{equation*}
        W(x_{k+1},\mu_{k+1})\le W(x_k,\mu_k)-\eta_1\frac{\Delta m_k}{\mu_k}\le W(x_k,\mu_k)-\eta_1\nu\kappa_c\epsilon.
    \end{equation*}
    Since $W\ge0$ by \eqref{varphi1} and $\mu_k\le\mu_{\rm max}$ by Lemma \ref{lemq1}, the number of such infeasible successful iterations is at most $W(x_0,\mu_0)/(\eta_1\nu\kappa_c\epsilon)=\mathcal{O}(\epsilon^{-1})$.
\qed
\end{proof}

Consequently, the infeasible successful iterations contribute only $\mathcal{O}(\epsilon^{-1})$ iterations, which is dominated by $\mathcal{O}(\epsilon^{-3/2})$ for $\epsilon<1$. Hence, in the sequel of the optimality analysis we restrict attention to the near-feasible phase $\|c(x_k)\|\le\kappa_c\epsilon$, in which the normal step obeys $\|n_k\|\le\kappa_{\rm n}\|c(x_k)\|\le\kappa_{\rm n}\kappa_c\epsilon$ by Assumption $(G4)$.

Similarly, if the current iterate $x_k$ is accepted, the function decreases at iterate
$x_k$ can also be estimated in the norm $\|g_{k+1}^Q\|$.

First, we analyze the continuity of $g_k^Q$. By Assumptions $(G1)-(G4)$, $g_k$, $H_k$, and $Q_k$ are Lipschitz continuous, and $g_k$, $H_k$, and $n_k$ are bounded. Based on the Lipschitz continuity of $c_k$, it can be deduced that $n_k$ is Lipschitz continuous. Thus, for any $k_1, k_2$ we have
\begin{eqnarray*}
\|g_{k_1}^Q-g_{k_2}^Q\| 
&\leq& \|Q^T_{k_1}(g_{k_1}+H_{k_1}n_{k_1})-Q^T_{k_1}(g_{k_2}+H_{k_2}n_{k_2})\|+\|Q^T_{k_1}(g_{k_2}+H_{k_2}n_{k_2}) \\
&&-Q^T_{k_2}(g_{k_2}+H_{k_2}n_{k_2})\| \\
&\leq& \|g_{k_1}-g_{k_2}\|+\|H_{k_1}\|\|n_{k_1}-n_{k_2}\|+\|H_{k_1}-H_{k_2}\|\|n_{k_2}\| \\
&&+\|Q^T_{k_1}-Q^T_{k_2}\|\|g_{k_2}+H_{k_2}n_{k_2}\| \\
&\leq& L_g\|x_{k_1}-x_{k_2}\|+M_{\rm H}\cdot L_n\|x_{k_1}-x_{k_2}\|+L_H\|x_{k_1}-x_{k_2}\|\cdot M_n \\
&&+L_N\|x_{k_1}-x_{k_2}\|\cdot M_{gBn} \\
&\leq& C_1 \|x_{k_1}-x_{k_2}\| + L_N\|x_{k_1}-x_{k_2}\|\cdot M_{gBn}\\
&\leq& \kappa_{\rm GL} \|x_{k_1}-x_{k_2}\|,
\end{eqnarray*}
where $\kappa_{\rm GL}=C_1+L_N\cdot M_{gBn}$; all constants appearing in this estimate ($L_g$, $L_n$, $L_H$, $L_N$, $M_n$, $M_{gBn}$, $C_1$, $\kappa_{\rm GL}$, and $\kappa_2,\kappa_3$ below) are independent of $\epsilon$ and of the iteration index $k$. Second, because $f$ and $c$ have Lipschitz continuous second derivatives (Assumption $(G1)$), the reduced-gradient map $x\mapsto g^Q(x)=Q(x)^T(g(x)+H(x)n(x))$ is differentiable with a Lipschitz continuous Jacobian along the iteration path. Expanding $g^Q(x_{k+1})$ at $x_k$ to second order, and observing that the directional derivative of $g^Q$ along the tangential step $p_k=Q_kp_k^Q$ equals $H_k^Qp_k^Q$ up to an $\mathcal{O}(\|c(x_k)\|)$ correction (the derivative of the basis $Q$ and of the normal step $n$ contribute only through the small feasibility residual), there exist constants $\kappa_2,\kappa_3>0$, independent of $\epsilon$, such that
\begin{equation}\label{eqsecondorder}
\| g_{k+1}^Q-g_k^Q-H_k^Qp_k^Q\| \leq \kappa_2\|d_k\|^2+\kappa_3\|c(x_k)\|\,\|d_k\|.
\end{equation}
In the near-feasible phase $\|c(x_k)\|\le\kappa_c\epsilon$ this reduces to $\kappa_2\|d_k\|^2+\kappa_3\kappa_c\epsilon\|d_k\|$, which is the key ingredient that keeps the constants below independent of $\epsilon$. We now begin the proof.

\begin{lemma}\label{lemfo}
     Assume that the hard case does not occur at $x_k$, Assumptions $(G1)-(G4)$ hold, the $k$-th iteration is successful, $x_k$ is near-feasible in the sense that $\|c(x_k)\|\le\kappa_c\epsilon$, and the reduced gradient is not small, i.e.\ $\|g_k^Q\|\ge\kappa_g\epsilon/2$. Then for a constant $\varsigma>0$ independent of $\epsilon$ we have the following function decrease
\begin{eqnarray}\label{eqfunct}
w(x_k,\mu_k)-w(x_k+d_k,\mu_k)
\geq \eta_1\kappa_{\rm btm}\frac{\sqrt{h_k}}{6}\frac{\|g_{k+1}^Q\|^\frac{3}{2}}{(\varsigma+\sqrt{h_k})^\frac{3}{2}}.
\end{eqnarray}
\end{lemma}

\begin{proof}
From \eqref{TKM1} we have $\| g_k^Q\| \leq \|H_k^Q+\sqrt{h_k}\|p_k^Q\| I\| \cdot \|p_k^Q\|$. Using \eqref{hk40} and \eqref{equniformub} we have $\sqrt{h_k}\|p_k^Q\|\le3\max\{M_{\rm H},\sqrt{M_{gQ}}\}\kappa_{\rm h}^{1/4}$, so that
\begin{eqnarray}\label{eqpH}
\|p_k^Q\| \geq \frac{\| g_k^Q\|}{\kappa_{\rm {pH}}}\ge\frac{\kappa_g\epsilon}{2\kappa_{\rm {pH}}},
\qquad \kappa_{\rm {pH}}:=M_{\rm H}+3\max\{M_{\rm H},\sqrt{M_{gQ}}\}\,\kappa_{\rm h}^{1/4},
\end{eqnarray}
where the last inequality uses $\|g_k^Q\|\ge\kappa_g\epsilon/2$. From \eqref{TKM1} (equivalently \eqref{eqrel}) we also have
\begin{eqnarray}\label{eqrel}
        \| g_k^Q+H_k^Qp_k^Q\| = \sqrt{h_k}\|p_k^Q\|^2.
    \end{eqnarray}
Using the second-order bound \eqref{eqsecondorder} in the near-feasible phase and \eqref{eqrel}, we obtain
    \begin{eqnarray*}
        \|g_{k+1}^Q\| &\leq& \| g_{k+1}^Q-g_k^Q-H_k^Qp_k^Q\| +\| g_k^Q+H_k^Qp_k^Q\|\\
        &\leq& \kappa_2\|d_k\|^2+\kappa_3\kappa_c\epsilon\|d_k\| +\sqrt{h_k}\|p_k^Q\|^2.
    \end{eqnarray*}
By the triangle inequality (indeed $n_k$ and $p_k=Q_kp_k^Q$ are orthogonal, so that $\|d_k\|^2=\|n_k\|^2+\|p_k^Q\|^2$) we have $\|d_k\|\le\|n_k\|+\|p_k^Q\|$; combined with $\|n_k\|\le\kappa_{\rm n}\|c(x_k)\|\le\kappa_{\rm n}\kappa_c\epsilon$ (near-feasibility and Assumption $(G4)$) and \eqref{eqpH}, we obtain $\|n_k\|\le\frac{2\kappa_{\rm n}\kappa_c\kappa_{\rm {pH}}}{\kappa_g}\|p_k^Q\|$, and hence
    \begin{eqnarray*}
        \|d_k\|\le \left(1+\frac{2\kappa_{\rm n}\kappa_c\kappa_{\rm {pH}}}{\kappa_g}\right)\|p_k^Q\| =: \kappa_{\rm cp}\|p_k^Q\|,
    \end{eqnarray*}
with $\kappa_{\rm cp}=1+\frac{2\kappa_{\rm n}\kappa_c\kappa_{\rm {pH}}}{\kappa_g}$ independent of $\epsilon$. Therefore, using \eqref{eqpH} once more (so that $\epsilon\le\frac{2\kappa_{\rm {pH}}}{\kappa_g}\|p_k^Q\|$),
    \begin{eqnarray*}
        \|g_{k+1}^Q\| &\leq& \kappa_2\kappa_{\rm cp}^2\|p_k^Q\|^2+\kappa_3\kappa_c\kappa_{\rm cp}\epsilon\|p_k^Q\| +\sqrt{h_k}\|p_k^Q\|^2\\
        &\leq& \left(\kappa_2\kappa_{\rm cp}^2+\frac{2\kappa_3\kappa_c\kappa_{\rm cp}\kappa_{\rm {pH}}}{\kappa_g}+\sqrt{h_k}\right)\|p_k^Q\|^2
        = (\varsigma+\sqrt{h_k})\|p_k^Q\|^2,
    \end{eqnarray*}
    where $\varsigma:=\kappa_2\kappa_{\rm cp}^2+\frac{2\kappa_3\kappa_c\kappa_{\rm cp}\kappa_{\rm {pH}}}{\kappa_g}$ is independent of $\epsilon$. Hence $\|p_k^Q\|^2\ge\frac{\|g_{k+1}^Q\|}{\varsigma+\sqrt{h_k}}$, so that $\|p_k^Q\|^3\ge\frac{\|g_{k+1}^Q\|^\frac{3}{2}}{(\varsigma+\sqrt{h_k})^\frac{3}{2}}$. Plugging this into \eqref{eqm} we can get that
\begin{eqnarray*}
w(x_k,\mu_k)-m(x_k,H_k,\mu_k,d_k)
\geq \kappa_{\rm btm}\frac{\sqrt{h_k}}{6}\frac{\|g_{k+1}^Q\|^\frac{3}{2}}{(\varsigma+\sqrt{h_k})^\frac{3}{2}}.
\end{eqnarray*}
Provided that the $k$-th iteration is successful, we have \eqref{eqfunct}. The lemma is proved.
\qed
\end{proof}

\begin{lemma}
    Assume that the hard case does not occur at $x_k$, Assumptions $(G1)-(G4)$ hold, the $k$-th iteration is successful, $x_k$ is near-feasible in the sense that $\|c(x_k)\|\le\kappa_c\epsilon$, and the reduced gradient is not small, i.e.\ $\|g_k^Q\|\ge\kappa_g\epsilon/2$. Then there exists a constant $\varsigma_C>0$ such that
\begin{equation}\label{eqfunc}
    w(x_k,\mu_k)-w(x_k+d_k,\mu_k) \geq \varsigma_C\|g_{k+1}^Q\|^{3/2},
\end{equation}
where $\varsigma_C=\eta_1\kappa_{\rm btm}\frac{\sqrt{h_{\min}}}{6}\frac{1}{(\varsigma+\sqrt{h_{\min}})^{3/2}}$.
\end{lemma}
\begin{proof}
    By Lemma \ref{lemfo} we have, for the near-feasible, gradient-dominated iterate $x_k$,
    \begin{equation*}
        w(x_k,\mu_k)-w(x_k+d_k,\mu_k)
        \geq \eta_1\kappa_{\rm btm}\frac{\sqrt{h_k}}{6}\frac{\|g_{k+1}^Q\|^{3/2}}{(\varsigma+\sqrt{h_k})^{3/2}}.
    \end{equation*}
    Since $h_k\in[h_{\min},\kappa_{\rm h}]$ with $h_{\min}>0$ and $\kappa_{\rm h}<+\infty$, the continuous function $\varphi(h):=\sqrt{h}\,(\varsigma+\sqrt{h})^{-3/2}$ attains a positive minimum $\varphi_{\min}>0$ on $[h_{\min},\kappa_{\rm h}]$. Replacing $\varphi(h_k)$ by $\varphi_{\min}$ and absorbing $\eta_1\kappa_{\rm btm}/6$ into $\varsigma_C$ gives \eqref{eqfunc}. Here we have taken the specific value $\varsigma_C=\eta_1\kappa_{\rm btm}\frac{\sqrt{h_{\min}}}{6}(\varsigma+\sqrt{h_{\min}})^{-3/2}$, which is a valid lower bound provided $\varphi$ is minimized at $h_{\min}$; since $\varphi$ is increasing on $[0,4\varsigma^2]$ and decreasing on $[4\varsigma^2,+\infty)$, this requires $\kappa_{\rm h}\le4\varsigma^2$. The general case, in which $\varphi_{\min}$ may instead be attained at $\kappa_{\rm h}$, is addressed in Remark \ref{rem:varphimin} below.
\qed
\end{proof}

\begin{remark}\label{rem:varphimin}
The preceding argument is valid as written only when $\varphi(h):=\sqrt{h}(\varsigma+\sqrt{h})^{-3/2}$ attains its minimum on $[h_{\min},\kappa_{\rm h}]$ at $h_{\min}$. Since $\varphi$ is increasing on $[0,4\varsigma^{2}]$ and decreasing on $[4\varsigma^{2},+\infty)$, the true minimum is $\varphi_{\min}=\min\{\varphi(h_{\min}),\varphi(\kappa_{\rm h})\}$. To keep the constant $\varsigma_C$ in \eqref{eqfunc} independent of $\epsilon$---which is what makes the subsequent complexity argument yield $O(\epsilon^{-3/2})$---one needs $\varphi(\kappa_{\rm h})=\Omega(1)$, equivalently $\kappa_{\rm h}=O(1)$ (or at most $\kappa_{\rm h}=O(\epsilon^{-q})$ with $q$ small enough that $\varphi(\kappa_{\rm h})=\Omega(1)$, together with an $\epsilon$-independent $\varsigma$). The uniform bound $\kappa_{\rm h}=O(1)$ therefore plays a structural role: it is not merely a logarithmic term in the complexity, but is needed to lower-bound both the gradient decrease \eqref{eqfunc} and the negative-curvature decrease \eqref{eqs} by $\Omega(\epsilon^{3/2})$. Establishing $\kappa_{\rm h}=O(1)$ rigorously requires a sharper model-error analysis in the near-feasible phase (the constraint-linearization term $\mu_k\|c(x_k+d_k)-c(x_k)-A_kd_k\|=\mathcal{O}(\mu_k\|d_k\|^2)$ must be dominated by the penalty-weighted normal-step decrease $\mu_k\Delta m_k^N$). We therefore state the $O(\epsilon^{-3/2})$ complexity bound under the explicit structural hypothesis $\kappa_{\rm h}=O(1)$; Subsection \ref{sec:hardcase} analyses the hard case but does not by itself establish this hypothesis, which is kept as a standing assumption throughout the complexity analysis. We keep this dependence explicit so that no constant is silently treated as $\epsilon$-independent.
\end{remark}

\begin{corollary}\label{corosucce}
     Assume that Assumptions $(G1)-(G4)$ hold and that $\kappa_{\rm h}$ is chosen so that the constants $\varsigma_C$ in \eqref{eqfunc} and $\kappa_{\rm f}$ below are independent of $\epsilon$ (see Remark \ref{rem:varphimin}). For any two successive successful iterates, without loss of generality, we denote them as $j$ and $j+1$, then exactly one of the following two cases will occur:\\
(1) if $x_{j+1}$ satisfies \eqref{eq0}, \eqref{eq1} and \eqref{eq2}, this means $x_{j+1}$ is a point satisfying Definition \ref{ss}.\\
(2) otherwise, $w(x_{j},\mu_j)-w(x_{j+2},\mu_j)\geq {\kappa_{\rm f}}\epsilon^{3/2}$.
\end{corollary}
\begin{proof}
    Suppose that the first case does not hold; then we must have
    $$\|g_{j+1}^Q \| \geq \kappa_g \epsilon ~\text{or}~ \lambda_{\min} (H_{j+1}^Q) \leq -\kappa_H \sqrt{\epsilon}.$$
    Since both $j$ and $j+1$ are successful iterations, applying \eqref{eqfunc} to the $j$-th iteration yields
    \begin{eqnarray*}
    w(x_j,\mu_j)-w(x_{j+1},\mu_j)\geq \varsigma_C\|g_{j+1}^Q\|^\frac{3}{2}.
    \end{eqnarray*}
    Moreover, because $\mu_j\leq\mu_{j+1}$ and $\|c(x)\|\geq 0$, we have $w(x_{j+1},\mu_j)\leq w(x_{j+1},\mu_{j+1})$ and $w(x_{j+2},\mu_{j+1})\geq w(x_{j+2},\mu_j)$; applying \eqref{eqs} to the $(j+1)$-th iteration, it follows that
    \begin{eqnarray*}
    w(x_{j+1},\mu_j)-w(x_{j+2},\mu_j)&\geq& w(x_{j+1},\mu_{j+1})-w(x_{j+2},\mu_{j+1})\\
    &\geq& -\frac{\eta_1\kappa_{\rm btm} \lambda_{\min}(H_{j+1}^Q)^3}{6h_{j+1}}
    \geq -\frac{\eta_1\kappa_{\rm btm} \lambda_{\min}(H_{j+1}^Q)^3}{6 \kappa_{\rm h}},
    \end{eqnarray*}
    where the last inequality uses $h_{j+1}\leq\kappa_{\rm h}$ from Lemma \ref{lemupper} and $\lambda_{\min}(H_{j+1}^Q)^3\leq 0$ in the current case. If $\|g_{j+1}^Q\|\geq\kappa_g\epsilon$, then
    \begin{eqnarray*}
    w(x_j,\mu_j)-w(x_{j+2},\mu_j)\geq \varsigma_C\|g_{j+1}^Q\|^\frac{3}{2}\geq \varsigma_C\kappa_g^{3/2}\epsilon^{3/2}.
    \end{eqnarray*}
    If instead $\lambda_{\min}(H_{j+1}^Q)\leq-\kappa_H\sqrt{\epsilon}$, then $\left(\lambda_{\min}(H_{j+1}^Q)\right)^3\leq-\kappa_H^3\epsilon^{3/2}$, and the above two bounds give
    \begin{eqnarray*}
    w(x_j,\mu_j)-w(x_{j+2},\mu_j)\geq \frac{\eta_1\kappa_{\rm btm}\kappa_H^3}{6\kappa_{\rm h}}\epsilon^{3/2}.
    \end{eqnarray*}
    Combining the two cases, we obtain
    \begin{eqnarray*}
    w(x_j,\mu_j)-w(x_{j+2},\mu_j)\geq {\kappa_{\rm f}} \epsilon^{3/2},
    \end{eqnarray*}
    where ${\kappa_{\rm f}} = \min \left\{\varsigma_C\kappa_g^{3/2},\frac{\eta_1\kappa_{\rm btm}\kappa_H^3}{6{\kappa_{\rm h}}}\right\}$, thus the second statement holds.
The two constants entering ${\kappa_{\rm f}}$ are independent of $\epsilon$ precisely because of Remark \ref{rem:varphimin}: the gradient branch uses the $\epsilon$-independent $\varsigma_C$ (for $\kappa_{\rm h}=O(1)$), while the negative-curvature branch uses $\frac{\eta_1\kappa_{\rm btm}\kappa_H^3}{6\kappa_{\rm h}}$, which is $\epsilon$-independent because $h_{j+1}\le\kappa_{\rm h}=O(1)$ and $\lambda_{\min}(H_{j+1}^Q)^3\le-\kappa_H^3\epsilon^{3/2}$. This is the place where the bound $\kappa_{\rm h}=O(1)$ is indispensable; it is the structural hypothesis under which this corollary (and hence the global complexity bound) is stated, and it is the reason the negative-curvature step is analyzed separately in Subsection \ref{sec:hardcase}.
\qed
\end{proof}

Now, we are ready to provide the global iteration complexity analysis of the Algorithm \ref{alg:asoda}. We define the following sets: the index sets of successful iterations $\mathcal{S}_j$ and unsuccessful iterations $\mathcal{U}_j$ for the given iteration $j$, that is,
\begin{equation}
    \label{eqsucandunsucsetj}
    \mathcal{S}_j: = \{k\leq j: {\rho_k\geq \eta_1}\} \quad \text{and} \quad \mathcal{U}_j:= \{k\leq j: \rho_k < \eta_1\}.
\end{equation}

\begin{lemma}
    For any $j\geq 0$, let $\mathcal{S}_j$ and $\mathcal{U}_j$ be defined in \eqref{eqsucandunsucsetj}, then
    \begin{equation}\label{eqbound unsuc by suc}
        \vert \mathcal{U}_j \vert \leq \frac{1}{\log \gamma_2}\left[\frac{1}{2}\log \frac{{\kappa_{\rm h}}}{h_{0}}+\vert \mathcal{S}_j \vert \log \frac{1}{\gamma_4}\right].
    \end{equation}
\end{lemma}

\begin{proof}
    First from the mechanism of the algorithm, we have
    \begin{equation}\label{eqrelationsuc}
        \gamma_4\sqrt{h_k} \leq \sqrt{h_{k+1}}, \quad k\in \mathcal{S}_j
    \end{equation}
    and
    \begin{equation}\label{eqrelationunc}
        \gamma_2\sqrt{h_k}\leq \sqrt{h_{k+1}},\quad k\in \mathcal{U}_j.
    \end{equation}
    There exists an upper bound for $h_k$ from the Lemma \ref{lemupper}, thus combining the above two inequalities we have that
    \begin{equation}\label{eqdeduce}
        \sqrt{h_{0}}\gamma_4^{\vert\mathcal{S}_j\vert}\gamma_2^{\vert\mathcal{U}_j\vert}\leq \sqrt{{\kappa_{\rm h}}},
    \end{equation}
    which is equivalent to
    \begin{equation}
        \label{eqdeducelog}
        \vert \mathcal{S}_j \vert \log \gamma_4+ \vert \mathcal{U}_j \vert \log \gamma_2 \leq \frac{1}{2}\log \frac{{\kappa_{\rm h}}}{h_{0}}.
    \end{equation}
    By rearranging terms in \eqref{eqdeducelog}, we proved this result.
\qed
\end{proof}

Next, we assume that $j_f$ is the first index such that the following holds,
 $$\| g_{j_f+1}^Q\| \leq \kappa_g \epsilon, \ \lambda_{\min}(H_{j_f+1}^Q)\geq -\kappa_H \sqrt{\epsilon}.$$
That is, the next iterate is already an $\epsilon$-approximated second-order stationary point,
we get an upper bound for the cardinality of the set $\vert \mathcal{S}_{j_f} \vert$ as follows.

\begin{lemma}
    The cardinality of the set $\mathcal{S}_{j_f}$ satisfies
    \begin{equation}
        \label{eqcard of suc}
        \vert \mathcal{S}_{j_f} \vert \leq O(\epsilon^{-3/2}).
    \end{equation}
\end{lemma}

\begin{proof}
    First, we denote $k^s_{j}$ as the $j$-th element in $\mathcal{S}_{j_f}$. That is 
    \begin{equation*}
    \mathcal{S}_{j_f} := \{ k^s_{1},k^s_{2},\ldots,k^s_{\vert \mathcal{S}_{j_f}\vert} \}.
    \end{equation*}
    Since $x_{k^s_{j+1}}$ is not yet an $\epsilon$-approximate second-order stationary point for any $k^s_{j+1}$ with $j+1<j_f$, Corollary \ref{corosucce} (case (2)) gives
    \begin{equation}\label{eqsuc dec ult}
        w(x_{k^s_j},\mu_{k^s_j})-w(x_{k^s_{j+2}},\mu_{k^s_j})\geq {\kappa_{\rm f}} \epsilon^{3/2}
        \qquad \text{for}\quad j=1,2,\ldots,|\mathcal{S}_{j_f}|-2.
    \end{equation}
    Let $\mathcal{J}:=\left\{j\in\{1,\ldots,|\mathcal{S}_{j_f}|-2\}: j \text{ odd}\right\}$, so that $|\mathcal{J}|=\lceil (|\mathcal{S}_{j_f}|-2)/2\rceil \ge |\mathcal{S}_{j_f}|/2-1$. Summing \eqref{eqsuc dec ult} over $j\in\mathcal{J}$ and telescoping (using $\mu_{k^s_j}\le\mu_{k^s_{j+2}}$ so that $w(x_{k^s_{j+2}},\mu_{k^s_j})\le w(x_{k^s_{j+2}},\mu_{k^s_{j+2}})$), we obtain
    \begin{align}
        \notag
        w(x_0,\mu_0)-w(x_{j_f+1},\mu_0) & \geq \sum_{j\in\mathcal{J}} \left[w(x_{k^s_j},\mu_{k^s_j})-w(x_{k^s_{j+2}},\mu_{k^s_j})\right]                                \\
                            & \geq \left(\frac{\vert \mathcal{S}_{j_f}\vert}{2}-1\right){\kappa_{\rm f}} \epsilon^{3/2}.        \label{eqtelscop}
    \end{align}
    From \eqref{varphi1} we have that 
    \begin{eqnarray*}
 w(x_{j_f+1},\mu_0)\geq f_{\text{min}}.
\end{eqnarray*}
    
    Note that ${\kappa_f}>0$ is a constant depending on the problem parameters such as $M$, thus we conclude that $$\vert \mathcal{S}_{j_f}\vert \leq \frac{2\,(w(x_0,\mu_0)-f_{\text{min}})}{{\kappa_{\rm f}}}\epsilon^{-3/2}+2 = O(\epsilon^{-3/2}).$$
    We complete the proof.
\qed
\end{proof}

\begin{theorem}\label{thm.global}
    Assume that the structural hypothesis $\kappa_{\rm h}=O(1)$ of Remark \ref{rem:varphimin} holds. Then the adaptive Algorithm (ASODA) takes $O\left(\epsilon^{-3/2}\right)$ iterations to achieve a point $x_k$ satisfying $\|c(x_k)\|\le\kappa_c\epsilon$, $\|g_k^Q\| \leq \kappa_g \epsilon$ and $\lambda_{\min}(H_k^Q) \geq -\kappa_H \sqrt{\epsilon}$.
\end{theorem}
\begin{proof}
    Note that $x_{j_f+1}$ is the point satisfying Definition \ref{ss}. The total iteration number is the sum of $\vert \mathcal{S}_{j_f} \vert$ and $\vert \mathcal{U}_{j_f} \vert$. The results are a simple combination of \eqref{eqcard of suc} and \eqref{eqbound unsuc by suc}, both of which hold under the stated hypothesis $\kappa_{\rm h}=O(1)$.
\qed
\end{proof}

Apart from the number of iterations consumed by the algorithm, we are also interested in the total number of GHMs solved by the algorithm. We may use a bisection procedure since $h_k$ is continuous in most cases. In particular, locating $h_k$ brings computations in solving extra GHMs in the order of $O(\log(\epsilon^{-1}))$. Interested readers may refer to Appendix D of Ref. \cite{he2025HSODF}. We will not elaborate on this here.

\begin{theorem}\label{thmcomplexity bisection procedure}
    At an iterate $x_k$ of the Algorithm, the number of iterations of the bisection procedure is
    \begin{equation}\label{eqcomplexity of bisection}
        O\left(\log \left(\frac{{\kappa_{\rm h}} U_{\rm \phi} U_{\rm H}}{h_{\min}\sigma_b}\right)\right),
    \end{equation}
    where $U_{\phi}$ and $U_{\rm H}$ denote uniform upper bounds on the reduced gradient and the reduced Hessian, respectively.
    If we set the inexactness of the bisection procedure $\sigma_b=\epsilon$, then the iteration bound becomes
    \begin{equation}
        \label{eqcomplexity of bisection 2}
        O\left(\log \left(\frac{{\kappa_{\rm h}} U_{\rm \phi} U_{\rm H}}{h_{\min}}\epsilon^{-1}\right)\right).
    \end{equation}
\end{theorem}

We arrive at the conclusion on the total number of GHMs solved in Algorithm \ref{alg:asoda}.
\begin{theorem}
    \label{corofinal complexity}
    Let $\mathcal K_{\psi}$ be the total number of evaluations to solve GHMs in Algorithm.
    Suppose the tolerance for the bisection procedure is set to $\sigma_b=\epsilon$. Then, before reaching an iterate satisfying \eqref{eq0}--\eqref{eq2}, we have the following bound:
    \begin{equation}\label{eqoracle calls}
        \mathcal K_{\psi} = O\left(\epsilon^{-3/2}\log(\epsilon^{-1})\right).
    \end{equation}
\end{theorem}
\begin{proof}
    It is directly implied by Theorem \ref{thm.global} and Theorem \ref{thmcomplexity bisection procedure}.
\qed
\end{proof}

\subsection{A Discussion on the Hard Case}\label{sec:hardcase}

In this Subsection, we supplement the previous global iteration complexity analyses by examining the convergence of Algorithm \ref{alghardcase} for the hard case. Recall that in the hard case $h_k$ becomes discontinuous, and the flexibility of GHM provides us with a remedy by appropriately perturbing $g_k^Q$; this allows us to achieve a successful iteration at $t=0$.

As described in Section \ref{sec2}, in Algorithm \ref{alghardcase} the base iterate $x_k$ is held fixed, and the counter $i=0,1,\dots$ enumerates the successive hard-case attempts; a quantity carrying the subscript $(k,i)$ is evaluated at the $i$-th attempt while the base point remains $x_k$. We set $\lambda_1 = \lambda_{\min}(H_k^Q)$. For each of the following iterates $i$, we perturb $g_{k,i}^Q$ based on the preceding $h_{k,i-1}$. When gradually increasing $h_{k,i}$, we use the same bisection method indexed by $j$ to find $\delta_{k,i,j}$. We show that it eventually produces a successful iteration, and that once an iterate is successful, no further hard-case treatment is needed.

\begin{lemma}\label{lemperturbedphi.succleadsto}
    Assume that $g_{k,i}^Q$ is perturbed in the following manner
        $$g_{k,i}^Q:= g_k^Q + \frac{\varsigma}{\gamma_3^{2} h_{k,i-1}}\lambda_1^2 v_{k}.$$
        Then $g_{k,i}^Q$ and $p_{k,i}^Q$ satisfy
        \begin{equation*}
            \|g_{k,i}^Q-g_k^Q\| \leq \varsigma \|p_{k,i}^Q\|^2.
        \end{equation*}
\end{lemma}

\begin{proof}
    Since the $(k,i)$-th iteration is successful, the optimality condition \eqref{TKM2} gives
    \begin{equation}
        H_{k}^Q+\sqrt{h_{k,i}}\|p_{k,i}^Q\| I\succeq 0. \notag
    \end{equation}
    The above inequality further implies that
    \begin{equation}\label{eqlarge stepsize}
        \|p_{k,i}^Q\|^2 \geq \frac{1}{h_{k,i}}\lambda_1^2 \geq \frac{1}{\gamma_3^2 h_{k,i-1}}\lambda_1^2.
    \end{equation}
    Therefore, the gap between $g_{k,i}^Q$ and $g_k^Q$ can be bounded as
    $$
        \|g_{k,i}^Q - g_k^Q\|  = \frac{{\varsigma}}{\gamma_3^{2} h_{k,i-1}}\lambda_1^2 \leq {\varsigma}\|p_{k,i}^Q\|^2.
    $$
\qed
\end{proof}

Next we show that when the hard case occurs, Algorithm \ref{alghardcase} produces a successful iterate. The proof below tracks explicitly the $h$-scaling $\|d_{k,i}\|\sim|\lambda_1|/\sqrt{h_{k,i}}$ of the tangential step and its coupling with the normal step $n_k$; this is what removes the gap in the earlier argument and shows that acceptance is controlled by the magnitude of the negative curvature rather than by $h_{k,i}$ itself.

\begin{theorem}\label{thmhardcase}
     Assume that Assumptions $(G1)-(G4)$ hold, and that the hard case is triggered at $x_k$ with \emph{genuine} negative curvature, i.e.\ $\lambda_1=\lambda_{\min}(H_k^Q)\le-\kappa_{\rm nc}$, where
     \begin{equation}\label{eqkappanc}
       \kappa_{\rm nc}:=6(1+\mu_{\max})\kappa_0,
     \end{equation}
     with $\mu_{\max}$ the uniform upper bound of $\{\mu_{k,i}\}$ established in Lemma \ref{lemq1}. Let the perturbation parameter satisfy $\varsigma\le\sqrt{h_{\min}}/6$. Then the trial step $d_{k,i}$ generated by Algorithm \ref{alghardcase} is accepted at the first trial, so that the hard case is resolved in a single outer iteration and $h_{k,i}\le\gamma_3^2 h_{k-1}$; in particular, the hard case does not force any further growth of the regularization parameter.
\end{theorem}
\begin{proof}
Similar to Lemma \ref{lemupper}, we have
\begin{eqnarray*}
        &&m(x_{k},H_{k},\mu_{k,i},d_{k,i})-w(x_k+d_{k,i},\mu_{k,i})\\
        &=&f(x_k)-f(x_k+d_{k,i})+g_{k,i}^Td_{k,i}+ \frac{1}{2}d_{k,i}^TH_kd_{k,i}\\
        &&+\frac{\sqrt{h_{k,i}}}{3}\|d_{k,i}\|^3-\mu_{k,i}\|c(x_k+d_{k,i})\|+\mu_{k,i}\|c(x_k)+A_kd_{k,i}\| \\
        &\geq& -\|g_{k,i}-g_k\| \|d_{k,i}\|+\frac{\sqrt{h_{k,i}}}{3}\|d_{k,i}\|^3-(1+\mu_{k,i})\kappa_{\rm 0}\|d_{k,i}\|^2.
    \end{eqnarray*}
    Since the perturbation in Algorithm \ref{alghardcase} is applied only to the reduced gradient in the null-space basis, i.e., $g_{k,i}-g_k=Q_k(g_{k,i}^Q-g_k^Q)$, and $Q_k$ has orthonormal columns, we have $\|g_{k,i}-g_k\|=\|g_{k,i}^Q-g_k^Q\|$. Combining this with Lemma \ref{lemperturbedphi.succleadsto} yields $\|g_{k,i}-g_k\|\leq\varsigma\|p^Q_{k,i}\|^2$, and hence
    \begin{eqnarray*}
        &&m(x_{k},H_{k},\mu_{k,i},d_{k,i})-w(x_k+d_{k,i},\mu_{k,i})\\
        &\geq& -{\varsigma} \|d_{k,i}\|\|p^Q_{k,i}\|^2+\frac{\sqrt{h_{k,i}}}{3}\|d_{k,i}\|^3-(1+\mu_{k,i})\kappa_{\rm 0}\|d_{k,i}\|^2,
    \end{eqnarray*}
Now, since $\|d_{k,i}\|>0$, we use the orthogonality of the normal and tangential components. Because $d_{k,i}=n_k+Q_k p^Q_{k,i}$ with $n_k$ in the range of $A_k^T$ and $Q_k p^Q_{k,i}$ in the null space of $A_k$, we have $\|d_{k,i}\|^2=\|n_k\|^2+\|p^Q_{k,i}\|^2\ge\|p^Q_{k,i}\|^2$, hence $\varsigma\|p^Q_{k,i}\|^2\le\varsigma\|d_{k,i}\|^2$. Substituting this into the right-hand side of the last inequality gives
    \begin{eqnarray*}
        &&m(x_{k},H_{k},\mu_{k,i},d_{k,i})-w(x_k+d_{k,i},\mu_{k,i})\\
        &\geq& -\varsigma\|d_{k,i}\|^3+\frac{\sqrt{h_{k,i}}}{3}\|d_{k,i}\|^3-(1+\mu_{k,i})\kappa_{\rm 0}\|d_{k,i}\|^2\\
        &=& \Big(\frac{\sqrt{h_{k,i}}}{3}-\varsigma\Big)\|d_{k,i}\|^3-(1+\mu_{k,i})\kappa_{\rm 0}\|d_{k,i}\|^2.
    \end{eqnarray*}
Since $h_{k,i}\ge h_{\min}$ and $\varsigma\le\sqrt{h_{\min}}/6\le\sqrt{h_{k,i}}/6$, we have $\frac{\sqrt{h_{k,i}}}{3}-\varsigma\ge\frac{\sqrt{h_{k,i}}}{6}$, so that, extracting the positive factor $\|d_{k,i}\|^2$,
{\small
    \begin{eqnarray*}
        m(x_k,H_k,\mu_{k,i},d_{k,i})-w(x_k+d_{k,i},\mu_{k,i})
        &\geq& \|d_{k,i}\|^2\Big(\frac{\sqrt{h_{k,i}}}{6}\|d_{k,i}\|-(1+\mu_{k,i})\kappa_0\Big).
    \end{eqnarray*}}
It remains to lower-bound $\|d_{k,i}\|$. The optimality condition \eqref{TKM2}, i.e.\ $H_k^Q+\sqrt{h_{k,i}}\|p^Q_{k,i}\|I\succeq0$, applied to the minimal eigenpair of $H_k^Q$, gives $\|p^Q_{k,i}\|\ge|\lambda_1|/\sqrt{h_{k,i}}$, exactly as in \eqref{eqlarge stepsize}. Combined with $\|d_{k,i}\|\ge\|p^Q_{k,i}\|$, this yields $\frac{\sqrt{h_{k,i}}}{6}\|d_{k,i}\|\ge\frac{|\lambda_1|}{6}$. Using also $\mu_{k,i}\le\mu_{\max}$ (Lemma \ref{lemq1}), we arrive at
{\small
    \begin{eqnarray*}
        m(x_k,H_k,\mu_{k,i},d_{k,i})-w(x_k+d_{k,i},\mu_{k,i})
        &\geq& \|d_{k,i}\|^2\Big(\frac{|\lambda_1|}{6}-(1+\mu_{\max})\kappa_0\Big)\;\ge\;0,
    \end{eqnarray*}}
where the final inequality is exactly the genuine-negative-curvature condition $|\lambda_1|\ge\kappa_{\rm nc}$ of \eqref{eqkappanc}. Hence $m(x_k,H_k,\mu_{k,i},d_{k,i})-w(x_k+d_{k,i},\mu_{k,i})\ge0$. Since the model and the merit function coincide at the trial point's origin, i.e.\ $m(x_k,H_k,\mu_{k,i},0)=w(x_k,\mu_{k,i})$, the ratio \eqref{rhok1} satisfies $\rho_{k,i}=1-\frac{w(x_k+d_{k,i},\mu_{k,i})-m(x_k,H_k,\mu_{k,i},d_{k,i})}{\Delta m_{k,i}}\ge1\ge\eta_1$; the step $d_{k,i}$ is therefore accepted at the first trial, and Algorithm \ref{alghardcase} terminates in a single outer iteration without increasing $h_{k,i}$ beyond $\gamma_3^2 h_{k-1}$.
\qed
\end{proof}

\begin{remark}\label{rem:hardcaseregime}
    The condition $|\lambda_1|\ge\kappa_{\rm nc}$ delineates the regime in which the perturbation is genuinely needed. When $0<|\lambda_1|<\kappa_{\rm nc}$, the reduced Hessian is already ``nearly'' positive semidefinite: the tangential step induced by the perturbation satisfies $\|p^Q_{k,i}\|\le|\lambda_1|/\sqrt{h_{k,i}}\le\kappa_{\rm nc}/\sqrt{h_{\min}}=O(1)$, so the hard case is benign and the iteration is governed either by the normal step (when $\|c(x_k)\|>\kappa_c\epsilon$) or by the reduced-gradient step (when $\|g_k^Q\|\ge\kappa_g\epsilon$); both cases are covered by the non-hard-case analyses of the preceding subsections, and no separate negative-curvature decrease is required. When the hard case is actually triggered ($t_{k,j}=0$) with $0<|\lambda_1|<\kappa_{\rm nc}$, Algorithm \ref{alghardcase} still perturbs $g_k^Q$ along the minimal eigenvector, so that $t_{k,i,j}\neq0$ and the tangential step $p^Q_{k,i}$ is well defined and bounded by $|\lambda_1|/\sqrt{h_{k,i}}=O(1)$; the trial step is then governed by the normal-step or reduced-gradient decrease and is covered by the analyses of the preceding subsections, without invoking the genuine-negative-curvature condition $|\lambda_1|\ge\kappa_{\rm nc}$ of Theorem \ref{thmhardcase}. The fact that $\kappa_{\rm nc}$ is a constant independent of $\epsilon$ is precisely what makes the negative-curvature branch of Corollary \ref{corosucce} contribute $\Omega(\epsilon^{3/2})$ per successful iteration under $\kappa_{\rm h}=O(1)$, as recorded in Remark \ref{rem:varphimin}.
\end{remark}

Indeed, Theorem \ref{thmhardcase} implies that the hard case is resolved within a single iteration of Algorithm \ref{alghardcase} in the genuine-negative-curvature regime, and reduces to the standard normal- or gradient-step analysis otherwise (see Remark \ref{rem:hardcaseregime}).

\section{Numerical Results}\label{sec4}

In this section, we present the numerical results of ASODA (Algorithm \ref{alg:asoda}). All experiments were carried out on a Dell laptop equipped with an Intel(R) Core(TM) i5-1035G1 CPU @ 1.00 GHz; the code was implemented in MATLAB and run under MATLAB(R2020b). The test set consists of equality-constrained problems from the CUTEst collection \cite{Gould2015cutest}; the $81$ problems on which ASODA reached the prescribed accuracy are reported in Table \ref{table5.1}. In the table, $n$ and $m$ are the numbers of variables and constraints, respectively; NI denotes the number of iterations, and NF, NC, and NG denote the numbers of function, constraint, and gradient evaluations, respectively; the column ``Res'' reports the final residual $\mathrm{Res}=\max\{\|g_k^Q\|,\|c(x_k)\|\}$, and ``time'' reports the CPU time in seconds. 
\emph{Experimental setup.} The implementation follows the construction of Section \ref{sec2}: at each iteration the normal step $n_k^c$ is the minimum-norm solution of $A_kn+c(x_k)=0$, the generalized homogeneous model \eqref{pkQ} is treated as an eigenvalue problem, and the scalar $\delta_{k,j}$ is updated by the bisection procedure, both as described in Section \ref{sec2}. All algorithmic parameters ($\delta_0$, $h_{\min}$, $\eta_1$, $\eta_2$, $\gamma_1$--$\gamma_4$, $\mu_0$, $\nu$, $\tau_1$, $\tau_2$, $\sigma_b$) are numerical constants independent of the problem data and of $\epsilon$, in accordance with the parameter-free design of the method; they are fixed numerical values that are never adapted to the tolerance $\epsilon$, and their values are those recorded in the accompanying code. ASODA is deterministic---no random initialization or randomized linear algebra is used---so a single run per problem is reported. Because the competing implementations differ in language and in the level of optimization, the comparison is made in terms of the evaluation counts NF, NG and NC, which are machine-independent; the CPU times in Table \ref{table5.1} are reported for ASODA only, as an indication of practical cost.

Within each comparison, all methods are run to the same accuracy target. For ASODA this target is the second-order stationarity condition of Definition \ref{ss}: the algorithm stops at the first iterate $x_k$ satisfying $\|c(x_k)\|\le\kappa_c\epsilon$, $\|g_k^Q\|\le\kappa_g\epsilon$, and $\lambda_{\min}(H_k^Q)\ge-\kappa_H\sqrt{\epsilon}$, where we take $\kappa_c=\kappa_g=\kappa_H=1$ for simplicity; this is exactly the stopping condition of the while-loop in Algorithm \ref{alg:asoda}. The competing implementations use their own stopping rules, but are driven to the same tolerance $\epsilon$ as ASODA in each comparison. The tolerance is $\epsilon=10^{-8}$ for Table \ref{table5.1}; for the comparisons of Tables \ref{table4} and \ref{table6} the tolerance of the respective cited paper is used (see below).

We compare Algorithm \ref{alg:asoda} (ASODA) with Algorithm 2.1 in \cite{ZhongwenChen}, LANCELOT, and Algorithms 1 and 2 in \cite{Arreckx}; the comparison results are presented in Table \ref{table4} and Table \ref{table6}. The comparison results show that, in most cases, ASODA requires fewer function, gradient, and constraint evaluations than its competitors, demonstrating competitive numerical performance.

\emph{A note on instance sizes across tables.} For certain test problems, the dimensions reported in Table \ref{table5.1} differ from those in Tables \ref{table4} and \ref{table6} (e.g.\ BRATU2D with $n=1024$, $m=900$ in Table \ref{table5.1} versus $n=484$, $m=400$ in Table \ref{table4}). This is because the competing methods were run on smaller discretizations of these problems, whereas Table \ref{table5.1} reports ASODA on the larger default instances. For a like-for-like comparison, ASODA was re-run on exactly the same instances used for the competing methods, and the counts reported for ASODA in Tables \ref{table4} and \ref{table6} correspond to those instances; hence all within-table comparisons are made on identical problem instances.

\emph{On the comparability of the tables.} Tables \ref{table4} and \ref{table6} report ASODA run on the same instances and at the same tolerance as the corresponding competitors ($\epsilon=10^{-5}$ and $\epsilon=10^{-6}$, respectively, as in the cited papers), which differ from the tighter tolerance $\epsilon=10^{-8}$ used in Table \ref{table5.1}. Consequently, the ASODA counts in these tables need not coincide with those in Table \ref{table5.1}. Within each table, all methods are compared on identical instances and at a common tolerance, so that the comparisons are like-for-like.

\emph{Test-set selection and failures.} The equality-constrained CUTEst problems reported here are those for which ASODA reached the prescribed tolerance within the iteration budget. Problems whose initial point already satisfies the termination criterion are not listed, since the algorithm terminates at $k=0$; this is the case, for instance, for the DRCAVTY1--3 problems, which are therefore absent from all three tables. In Table \ref{table4} the symbol ``*'' marks a problem for which LANCELOT did not reach the required accuracy within the prescribed budget (FLOSP2TL), which ASODA solved without difficulty; this is the only entry in which a competing method failed, so that the ASODA--LANCELOT comparison in Table \ref{table4} covers $42$ problems whereas the ASODA--Algorithm~2.1 comparison covers all $43$ rows of that table.

\begin{center}
\small
\setlength{\tabcolsep}{4pt}
{
\begin{longtable}{ccccccccc}  
\caption{Numerical results of the ASODA Algorithm}\label{table5.1}
  \endfirsthead   
  \multicolumn{9}{c}{Table \ref{table5.1} continued}\\
  \hline
\multirow{2}{*}{Problem}  & \multicolumn{2}{c}{Dimension} &\multirow{2}{*}{NI}   & \multirow{2}{*}{NF}  &\multirow{2}{*}{NC}  &\multirow{2}{*}{NG}  &\multirow{2}{*}{Res}  &\multirow{2}{*}{time}\\
                          \cmidrule{2-3}

                          &$n$&$m$                                                         \\
  \hline
  \endhead
  \hline
\multirow{2}{*}{Problem}  & \multicolumn{2}{c}{Dimension} &\multirow{2}{*}{NI}   & \multirow{2}{*}{NF}  &\multirow{2}{*}{NC}  &\multirow{2}{*}{NG}  &\multirow{2}{*}{Res}  &\multirow{2}{*}{time}\\
                          \cmidrule{2-3}

                          &$n$&$m$                       \\
\hline
AIRCRFTA & 8 & 5 & 3 & 3 & 3 & 3 & 7.0525e-19 & 0.028331 \\
ARTIF & 102 & 100 & 12 & 12 & 12 & 12 & 1.6079e-11 & 0.262716 \\
BAmL1SP & 57 & 12 & 5 & 5 & 5 & 5 & 5.6454e-12 & 2.281692	\\
BRATU2D & 1024 & 900 & 3 & 3 & 3 & 3 & 2.6208e-09 & 1.907733 \\
BRATU2DT & 1024 & 900 & 4 & 4 & 4 & 4 & 1.9376e-14 & 2.459303 \\
BRATU3D & 4913 & 3375 & 4 & 4 & 4 & 4 & 4.3241e-14 & 77.653837 \\
BT1 & 2 & 1 & 6 & 6 & 6 & 6 & 3.3878e-10 & 0.031558\\
BT2 & 3 & 1 & 15 & 15 & 15 & 15 & 1.8443e-09 & 0.100351\\
BT3 & 5 & 3 & 14 & 14 & 14 & 14 & 2.0958e-09 & 0.035366\\
BT4 & 3 & 2 & 7 & 7 & 7 & 7 & 3.4734e-09 & 0.032733\\
BT5 & 3 & 2 & 5 & 5 & 5 & 5 & 4.6573e-09 & 0.031506\\
BT6 & 5 & 2 & 17 & 17 & 17 & 17 & 7.0832e-09 & 0.134826\\
BT8 & 5 & 2 & 7 & 7 & 7 & 7 & 6.6637e-09 & 0.073569\\
BT9 & 4 & 2 & 10 & 10 & 10 & 10 & 9.6644e-09 & 0.046409\\
BT11 & 5 & 3 & 7 & 7 & 7 & 7 & 7.9468e-09 & 0.032276\\
BT12 & 5 & 3 & 12 & 12 & 12 & 12 & 4.1558e-09 & 0.061319\\
BYRDSPHR & 3 & 2 & 10 & 10 & 10 & 10 & 3.8218e-09 & 0.032310 \\
CBRATU2D & 3200 & 2888 & 4 & 4 & 4 & 4 & 1.1170e-15 & 28.894635\\
CBRATU3D & 3456 & 2000 & 4 & 4 & 4 & 4 & 3.8812e-14 & 29.493855\\
DTOC2 & 5998 & 3996 & 6 & 6 & 6 & 6 & 8.5438e-09 & 444.684008	\\
DTOC3 & 299 & 198 & 22 & 22 & 22 & 22 & 6.8839e-10 & 1.715504	\\
DTOC4 & 2999 & 1998 & 13 & 13 & 13 & 13 & 4.1047e-09 & 110.130702	\\
DTOC5 & 1999 & 999 & 12 & 12 & 12 & 12 & 5.0670e-09 & 29.892085	\\
DTOC6 & 201 & 100 & 23 & 23 & 23 & 23 & 1.7032e-08 & 3.601708	\\
DTOC1NA & 598 & 396 & 19 & 19 & 19 & 19 & 7.6965e-10 & 16.896	\\
DTOC1NB & 598 & 396 & 17 & 17 & 17 & 17 & 2.3243e-10 & 17.290252	\\
DTOC1NC & 598 & 396 & 10 & 10 & 10 & 10 & 7.5235e-09 & 8.895817	\\
DTOC1ND & 598 & 396 & 11 & 11 & 11 & 11 & 1.5045e-10 & 10.671473	\\
DTOC1L & 298 & 196 & 12 & 12 & 12 & 12 & 6.8506e-08 & 1.250889 \\
EIGENA2 & 6 & 3 & 4 & 4 & 4 & 4 & 8.8540e-10 & 0.029693 \\
EIGENB2 & 6 & 3 & 10 & 10 & 10 & 10 & 8.0625e-09 & 0.086322 \\
FLOSP2TL & 363 & 323 & 5 & 5 & 5 & 5 & 9.4773e-12 & 0.414656\\
GENHS28 & 10 & 8 & 15 & 15 & 15 & 15 & 3.8322e-09 & 0.044049	\\
HAGER1 & 2001 & 1000 & 4 & 4 & 4 & 4 & 7.1809e-09 & 11.948625	\\
HAGER2 & 1001 & 500 & 7 & 7 & 7 & 7 & 3.6617e-09 & 18.388018	\\
HAGER3 & 2001 & 1000 & 7 & 7 & 7 & 7 & 4.8667e-09 & 19.755394	\\
HS06 & 2 & 1 & 10 & 10 & 10 & 10 & 1.2508e-09 & 0.032876	\\
HS07 & 2 & 1 & 9 & 9 & 9 & 9 & 5.3501e-09 & 0.072297	\\
HS09 & 2 & 1 & 18 & 18 & 18 & 18 & 4.4672e-09 & 0.04275	\\
HS26 & 3 & 1 & 25 & 25 & 25 & 25 & 3.4531e-09 & 0.159187	\\
HS27 & 3 & 1 & 28 & 28 & 28 & 28 & 7.0892e-09 & 0.040849\\
HS28 & 3 & 1 & 10 & 10 & 10 & 10 & 1.6656e-09 & 0.036106	\\
HS39 & 4 & 2 & 11 & 11 & 11 & 11 & 6.8425e-10 & 0.105259	\\
HS40 & 4 & 3 & 5 & 5 & 5 & 5 & 4.5759e-09 & 0.042367	\\
HS42 & 4 & 2 & 4 & 4 & 4 & 4 & 3.3342e-09 & 0.019618	\\
HS46 & 5 & 2 & 19 & 19 & 19 & 19 & 1.8829e-09 & 0.123811	\\
HS47 & 5 & 3 & 18 & 18 & 18 & 18 & 9.4240e-09 & 0.054164	\\
HS48 & 5 & 2 & 11 & 11 & 11 & 11 & 9.2041e-09 & 0.022006	\\

HS51 & 5 & 3 & 7 & 7 & 7 & 7 & 5.4843e-10 & 0.038322	\\
HS56 & 7 & 4 & 21 & 21 & 21 & 21 & 8.4754e-09 & 0.041807	\\
HS61 & 3 & 2 & 9 & 9 & 9 & 9 & 3.2695e-10 & 0.034174	\\
HS78 & 5 & 3 & 13 & 13 & 13 & 13 & 5.0439e-09 & 0.087216	\\
HS79 & 5 & 3 & 5 & 5 & 5 & 5 & 1.6753e-09 & 0.037794	\\
HS100LNP & 7 & 2 & 24 & 24 & 24 & 24 & 9.3762e-09 & 0.226268	\\
HS111LNP & 10 & 3 & 28 & 28 & 28 & 28 & 2.8635e-09 & 0.069632\\
HUESTIS & 100 & 2 & 28 & 28 & 28 & 28 & 2.3408e-11 & 1.108031	\\
HUESmMOD & 10 & 2 & 25 & 25 & 25 & 25 & 3.3878e-11 & 0.600084\\
INTEGREQ & 502 & 500 & 4 & 4 & 4 & 4 & 3.1781e-15 & 91.544464	\\
MARATOS & 2 & 1 & 4 & 4 & 4 & 4 & 2.9091e-09 & 0.020958	\\
MWRIGHT & 5 & 3 & 30 & 30 & 30 & 30 & 3.7061e-09 & 0.11604\\
ORTHRDM2 & 2003 & 1000 & 13 & 13 & 13 & 13 & 6.5672e-12 & 220.341976	\\
ORTHREGB & 27 & 6 & 4 & 4 & 4 & 4 & 3.2755e-09 & 0.404811\\
ORTHREGC & 1005 & 500 & 30 & 30 & 30 & 30 & 6.0319e-09 & 60.796477	\\
ORTHRGDM & 1003 & 500 & 12 & 12 & 12 & 12 & 1.5631e-09 & 24.373448	\\
ORTHRGDS & 503 & 250 & 17 & 17 & 17 & 17 & 7.0411e-09 & 9.256148	\\
POROUS1 & 4096 & 3844 & 6 & 6 & 6 & 6 & 9.9518e-11 & 116.184786\\
POROUS2 & 4096 & 3844 & 7 & 7 & 7 & 7 & 9.2658e-11 & 136.759611\\
S219 & 4 & 2 & 22 & 22 & 22 & 22 & 4.2471e-09 & 0.149621	\\
S235 & 3 & 1 & 14 & 14 & 14 & 14 & 5.0599e-09 & 0.041762	\\
S252 & 3 & 1 & 15 & 15 & 15 & 15 & 7.1699e-09 & 0.060256	\\
S254 & 3 & 2 & 7 & 7 & 7 & 7 & 6.1284e-10 & 0.035348\\
S269 & 5 & 3 & 6 & 6 & 6 & 6 & 1.4069e-09 & 0.033374\\
S317 & 2 & 1 & 9 & 9 & 9 & 9 & 2.6413e-09 & 0.045241\\
S318 & 2 & 1 & 12 & 12 & 12 & 12 & 3.9550e-10 & 0.043701\\
S321 & 2 & 1 & 20 & 20 & 20 & 20 & 6.8505e-09 & 0.049496\\
S344 & 3 & 1 & 8 & 8 & 8 & 8 & 4.7680e-09 & 0.080072\\
S378 & 10 & 3 & 26 & 26 & 26 & 26 & 6.5821e-09 & 0.057716	\\
S394 & 20 & 1 & 25 & 25 & 25 & 25 & 6.5976e-09 & 0.06195	\\
SEMICN2U & 12 & 10 & 11 & 11 & 11 & 11 & 8.3457e-09 & 0.073321\\
TAME & 2 & 1 & 1 & 1 & 1 & 1 & 3.3307e-16 & 0.009552	\\
TATP1CNE & 2600 & 2600 & 5 & 5 & 5 & 5 & 1.7154e-12 & 132.101433\\

\hline
\end{longtable}
}
\end{center}

Several observations can be drawn from Table \ref{table5.1}. First, regarding the convergence behavior, ASODA reaches the prescribed accuracy on all $81$ problems in at most $30$ iterations; apart from TAME, which terminates at $k=1$, every problem requires at least $3$ iterations, so that the iteration count lies in $[3,30]$ for $80$ of the $81$ problems. The count shows no systematic growth with the problem dimension: large-scale instances such as BRATU3D ($n=4913$), DTOC2 ($n=5998$) and POROUS1 ($n=4096$) are solved within $4\sim6$ iterations, whereas the problems with $n\le 10$ need $10\sim30$ iterations in about $60\%$ of the cases, with a median of $11$. This absence of dimension dependence is a hallmark of second-order methods: the incorporation of curvature information makes the iteration count essentially independent of the problem size, in contrast to first-order methods, whose iteration counts are governed mainly by the conditioning of the problem rather than by its dimension.

Second, the magnitude of the reported residual Res deserves particular attention. For $79$ of the $81$ problems the final residual satisfies $\mathrm{Res}\le\epsilon$: $52$ of them terminate in $[10^{-9},10^{-8})$ and a further $16$ between $10^{-11}$ and $10^{-9}$, so that in the majority of cases the termination test is met just below the tolerance rather than far below it. A smaller group of problems terminates well below the tolerance---for instance CBRATU2D ($1.1\times10^{-15}$), INTEGREQ ($3.2\times10^{-15}$) and TAME ($3.3\times10^{-16}$)---and AIRCRFTA attains $7.1\times10^{-19}$, close to machine precision. The two remaining problems, DTOC6 and DTOC1L, terminate marginally above the tolerance, at $1.7\times10^{-8}$ and $6.9\times10^{-8}$ respectively. 
For the problems that do terminate well below the tolerance, this indicates empirically rapid terminal convergence near the solution (no formal local rate is claimed in this work). This is consistent with the mechanism of the adaptive cubic-regularization term: once the iterates enter a small neighborhood of a solution, the successful iterations shrink the interval $I_h$ and hence the coefficient $\sqrt{h_k}$, so that the cubic term $\frac{\sqrt{h_k}}{3}\|d_k\|^3$ becomes negligible and the trial step approaches a Newton step, yielding the observed fast terminal convergence.

Third, the results are consistent with the global complexity bound $\mathcal{O}(\epsilon^{-3/2})$ established in Section \ref{sec3}. The theory guarantees that the worst-case number of iterations is $\mathcal{O}(\epsilon^{-3/2})$; the observed iteration counts are far below this worst-case bound, which is expected, because the bound characterizes the global worst-case behavior over all admissible functions, whereas the CUTEst problems are benign and admit fast local convergence. Moreover, each iteration solves exactly one generalized homogeneous model (up to the additional eigensolves generated by the bisection on $\delta_{k,j}$), so the number of such solves is of the same order as the iteration count, in agreement with Theorem \ref{corofinal complexity}; the hard case (Algorithm \ref{alghardcase}) is handled by the perturbation procedure, which by Theorem \ref{thmhardcase} resolves it within a single outer iteration in the genuine-negative-curvature regime. 

\begin{center}
\small
\setlength{\tabcolsep}{1.8pt}
{
		\begin{longtable}{ccccccccccccc}
			\caption{Comparison results 1}\label{table4}
			\endfirsthead
			\multicolumn{13}{l}{Table \ref{table4} (continued)}\\
			\hline
			\multirow{2}{*}{Problem}  & \multicolumn{2}{l}{Dimension}&&\multicolumn{2}{c}{ASODA }&&\multicolumn{2}{c}{Alg.2.1 in Ref. \cite{ZhongwenChen}}&&\multicolumn{2}{c}{LANCELOT } \\
			\cmidrule{2-3}
			\cmidrule{5-6}
            \cmidrule{8-9}
            \cmidrule{11-12}
            &$n$&$m$                       & &NF&NC   & &NF&NC   &&NF&NC& \\
			\hline
			\endhead
			\hline
			\multirow{2}{*}{Problem}  & \multicolumn{2}{l}{Dimension}&&\multicolumn{2}{c}{ASODA }&&\multicolumn{2}{c}{Alg.2.1 in Ref. \cite{ZhongwenChen}}&&\multicolumn{2}{c}{LANCELOT } \\
		    \cmidrule{2-3}
			\cmidrule{5-6}
            \cmidrule{8-9}
            \cmidrule{11-12}
			&$n$&$m$                       & &NF&NC   & &NF&NC   &&NF&NC& \\
			\hline
 AIRCRFTA	&8	&5	&&2&2 &&12	&21 &&5& 5\\
BRATU2D	    &484	&400	&&3&3  &&6&10 &&3& 3\\
BRATU2DT	&484	&400	&&3	&3 &&6&10 &&8 &8\\
BRATU3D	    &125	&27		&&3	&3 &&10&18 &&5& 5\\
BT1	        &2	&1	&&6	&6	&&8&9 &&48 &48\\
BT2	        &3	&1	&&12&12	&&11&11 &&22& 22\\
BT3	        &5	&3	&&14&14	&&6&6 &&16 &16\\
BT4	        &3	&2	&&7	&7	&&9&9 &&28& 28\\
BT5	        &3	&2	&&5	&5	&&5&5 &&16& 16\\
BT6	        &5	&2	&&11&11	&&14&14 &&26 &26\\
BT8	        &5	&2	&&7	&7	&&12&19 &&25& 25\\
BT9	        &4	&2	&&10&10	&&13&15 &&20& 20\\
BT11	    &5	&3	&&7	&7	&&8&8 &&19 &19\\
BT12	    &5	&3	&&12&12	&&7&8 &&21& 21\\
BYRDSPHR	&3	&2	&&10&10	&&10&13 &&22& 22\\
CBRATU2D	&98	&50&&2&2 &&9&16&& 5& 5\\
CBRATU3D	&128&16&&3&3&&4&4 &&5 &5\\
DTOC2	&598&396&&4&4&&17&27 &&21& 21	\\
DTOC4		&299&198 &&4&4&&3&3&& 17& 17\\
DTOC5	&19	&9&&3&3&&4&4 &&14 &14\\
EIGENA2	&6	&3	&&4	&4&&3&3 &&7 &7\\
EIGENB2	&6	&3	&&10&10	&&9&9 &&39& 39\\
FLOSP2TL	&363&323&&5&5&&9&9 &&* &*\\
HAGER1	&1001	&500&&2&2&&2&2 &&13& 13\\
HAGER2	&1001	&500&&2&2&&6&6 &&12& 12\\
HAGER3	&1001	&500&&2&2&&14&14 &&13 &13\\
HS06	&2	&1	&&10&10&&10&13 &&30 &30\\
HS07	&2	&1	&&6	&6	&&9&11 &&17& 17\\
HS26	&3	&1	&&18	&18	&&19&19&& 22& 22\\
HS27	&3	&1	&&16&16 &&23&25 &&14 &14\\
HS39	&4	&2	&&8	&8	&&13&15 &&20& 20\\
HS40	&4	&3	&&5	&5	&&4&4 &&16 &16\\
HS42	&4	&2	&&4	&4	&&4&4 &&13& 13\\	
HS46	&5	&2	&&16&16	&&17&17 &&19& 19\\
HS47	&5	&3	&&15&15	&&18&18 &&21 &21\\		
HS51	&5	&3	&&7	&7  &&2&2 &&7 &7\\
HS56	&7	&4	&&14&14	&&9&9 &&18& 18\\
HS61	&3	&2	&&9	&9	&&6&6 &&16 &16\\
HS78	&5	&3	&&8	&8	&&5&5 &&12& 12\\
HS79	&5	&3	&&5	&5	&&5&5 &&11& 11\\
INTEGREQ	&102&100&&2	&2 &&2&2&& 4& 4\\
MARATOS	    &2	&1	&&4	&4	&&4&4 &&9& 9\\
ORTHREGB	&27	&6	&&4	&4	&&25&30 &&74& 74\\	
			\hline
		\end{longtable}
	}
\end{center}

\begin{center}
\footnotesize
	{
\setlength{\tabcolsep}{1.5pt}
		\begin{longtable}{cccccccccccccccccc}
			\caption{Comparison results 2}\label{table6}
			\endfirsthead
			\multicolumn{18}{l}{Table \ref{table6} (continued)}\\
			\hline
			\multirow{2}{*}{Problem}  & \multicolumn{2}{l}{Dimension}&&\multicolumn{3}{c}{ASODA }&&\multicolumn{3}{c}{Alg.1 in Ref. \cite{Arreckx}}&&\multicolumn{3}{c}{Alg.2 in Ref. \cite{Arreckx}} \\
			\cmidrule{2-3}
			\cmidrule{5-7}
            \cmidrule{9-11}
            \cmidrule{13-15}
            &$n$&$m$                       & &NF&NG&NC   & &NF&NG&NC   &&NF&NG&NC& \\
			\hline
			\endhead
			\hline
			\multirow{2}{*}{Problem}  & \multicolumn{2}{l}{Dimension}&&\multicolumn{3}{c}{ASODA }&&\multicolumn{3}{c}{Alg.1 in Ref. \cite{Arreckx}}&&\multicolumn{3}{c}{Alg.2 in Ref. \cite{Arreckx}} \\
		    \cmidrule{2-3}
			\cmidrule{5-7}
            \cmidrule{9-11}
            \cmidrule{13-15}
			&$n$&$m$                       & &NF&NG&NC   & &NF&NG&NC   &&NF&NG&NC& \\
			\hline
AIRCRFTA &8	&5&& 2&2 &2 &&5&3&5&& 18&15&18\\
BRATU3D	&4913	&3375 &&4&4&4 &&	7&	4	&7 &&99&63&	99\\
BT1	&2&	1 &&6&6&6 &&	17&	9&	17&& 1979	&1005	&1980\\
BT2	&3&	1 &&15&15&15 &&	19&	10&	19 &&35&35&	35\\
BT3	&5&	3 &&14&14& 14 &&	73&	37&	73 &&65&59&	65\\
BT4	&3&	2 &&7&7&7 &&	57&	29&	57 &&2009&1008	&2010\\
BT5&	3&	2 &&5&5& 5&	&17&	9&	17 &&69&61	&69\\
BT6	&5&	2 &&14&14& 14&&	23&	12&	23 &&2719	&239	&2719\\
BT8	&5&	2 &&7&7&7 &&	99&	40&	99 &&27	&27	&27\\
BT9	&4&	2 &&10&10& 10&&	25&	13&	25 &&21&21	&21\\
BT11&	5&	3 &&7&7& 7&	&35&	18&	35 &&30&30	&30\\
BT12&	5&	3 &&12&12&12&&	11&	6&	11 &&36	&36	&36\\
BYRDSPHR&	3&2  &&10&10&10 &	&35	&18&35 &&56&54	&56\\
CBRATU2D&	1058&	882 &&3&3&3 &&	3&	2&	3 &&90&50	&90\\
CBRATU3D&	2000&	1024 &&4&4& 4&&	3	&2&	3 &&17&16	&17\\
DTOC2	&5994&	3996 &&5&5& 5&&	45&	23&	45 &&981&510	&986\\
GENHS28	&10	&8&&15&15&15	&&17&	9&	17 &&64	&56	&64\\
HS6&	2&	1& &10& 10& 10&&17&	9&	17 &&14	&14	&14\\
HS7	&   2&	1& &9 &9& 9&&	17&	9&	17 &&21	&21	&21\\
HS9&	2&	1& &15 &15& 15&&	21&	11&	21 &&41	&21	&41\\
HS26&	3&	1 &&21& 21& 21&&	33&	17&	33 &&482	&162	&482\\
HS27&	3&	1&& 20 &20 &20&&	111&	51&	111 &&310	&88	&310\\
HS28&	3&	1&& 10 &10 &10&&	5&	3&	5 &&28	&28	&28\\
HS39&	4&	2&& 8& 8& 8&& 27&	14&	27 &&21	&21	&21\\
HS40&	4&	3 &&5& 5& 5&	&11&	6&	11 &&34&32	&34\\
HS46&	5&	2 &&23& 23& 23&&	49&	25&	49 &&597	&112	&597\\
HS47&	5&	3 &&16& 16& 16&	&95&	48&	95 &&625	&125	&625\\
HS48&	5&	2 &&11& 11& 11&&	5&	3&	5 &&38&38	&38\\
HS49&	5&	2 &&28 &28& 28&&	53&	27&	53 &&93	&76	&93\\
HS51&	5&	3&& 7& 7& 7&&	13&	7	&13 &&38	&38	&38\\
HS61&	3&	2 &&9& 9& 9&&	25&	13&	25 &&29	&29	&29\\
HS78&	5&	3& &9& 9& 9&&	43&	22&	43 &&28	&28	&28\\
HS79&	5&	3 &&5& 5& 5&&	17&	9&	17 &&27	&27	&27\\
HS100LNP&	7&	2 &&20&20&20 &&	17&	9&	17 &&119	&79	&119\\
HS111LNP	&10	&3 &&27&27&27	&&27&	14&	27 &&61	&52	&61\\
INTEGREQ	&102&	100 &&3&3&3 &&	5&	3&	5 &&39&	37&	39\\
MWRIGHT	&5	&3 &&24&24&24 &&	53&	27&	53 &&93&69&	93\\
ORTHREGB	&27&	6 &&4&4&4 &&	5&	3&	5 &&8	&8&8\\
			\hline
		\end{longtable}
	}
\end{center}

\begin{figure}[t]
\centering
\subfigure[Performance profiles based on NF]{
\begin{minipage}[t]{0.49\textwidth}
\centering
\includegraphics[height=5cm,width=5.8cm]{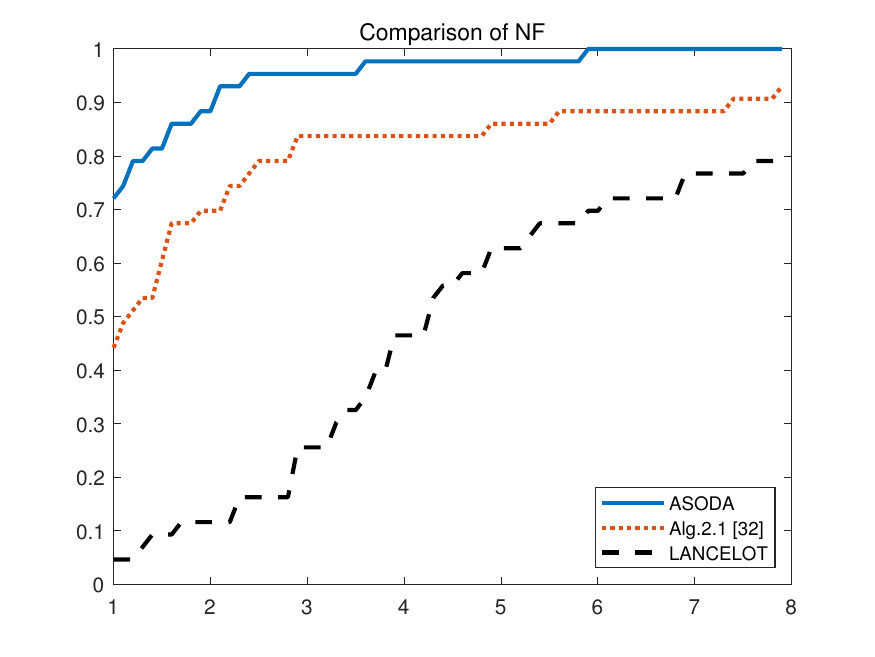}
\end{minipage}%
}%
\subfigure[Performance profiles based on NC]{
\begin{minipage}[t]{0.49\textwidth}
\centering
\includegraphics[height=5cm,width=5.8cm]{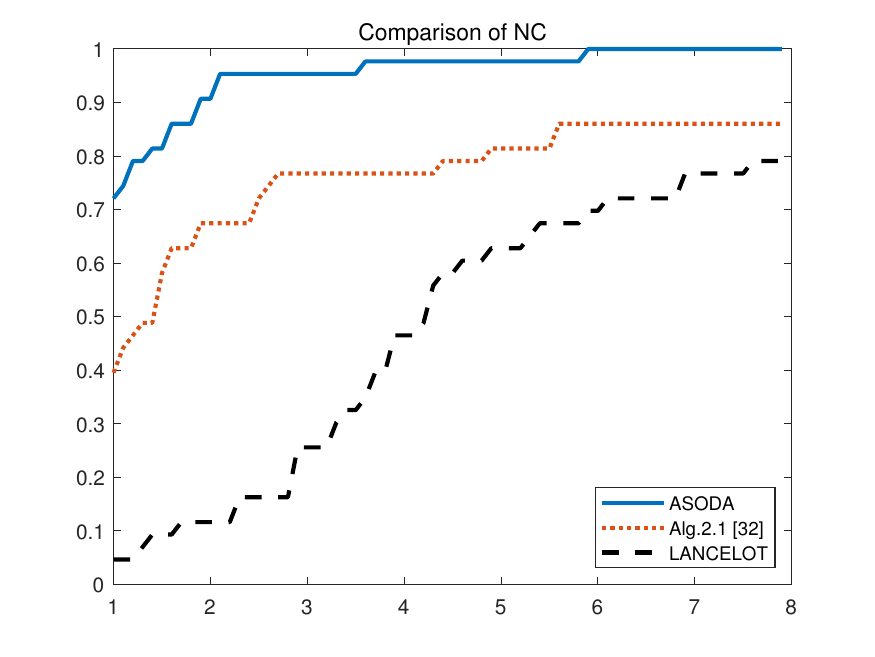}
\end{minipage}%
}%
\centering
\caption{Performance profiles comparing ASODA with Algorithm 2.1 in \cite{ZhongwenChen} and LANCELOT. 
The profiles show the fraction of problems solved within a factor $\tau$ of the best solver for (a) function evaluations and (b) constraint evaluations. 
ASODA attains the fewest function evaluations on 30 of the 42 problems shown.}\label{Figure1}
\end{figure}

\begin{figure}[t]
\centering
\subfigure[Performance profiles based on NF]{
\begin{minipage}[t]{0.49\textwidth}
\centering
\includegraphics[height=5cm,width=5.8cm]{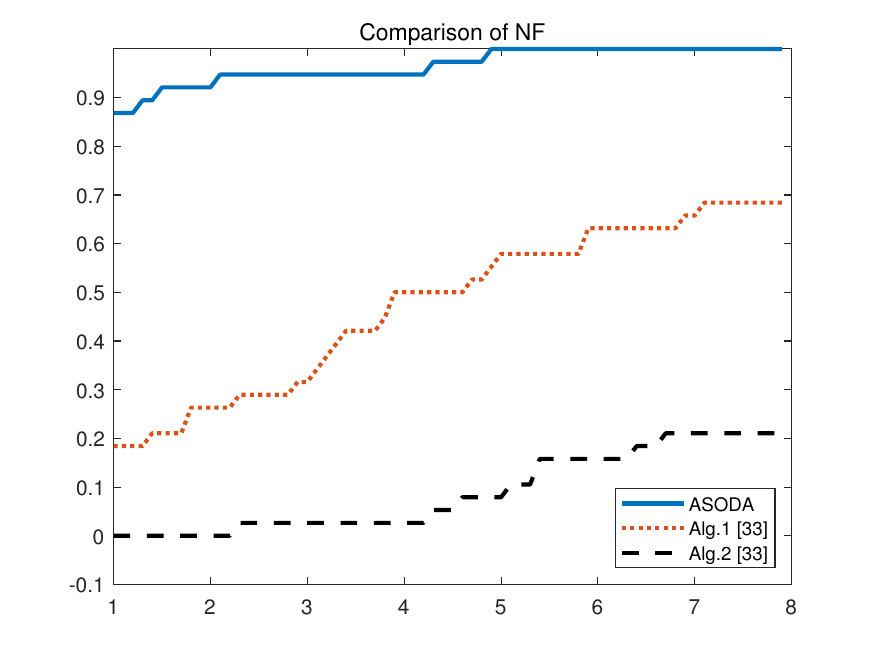}
\end{minipage}%
}%
\subfigure[Performance profiles based on NG]{
\begin{minipage}[t]{0.49\textwidth}
\centering
\includegraphics[height=5cm,width=5.8cm]{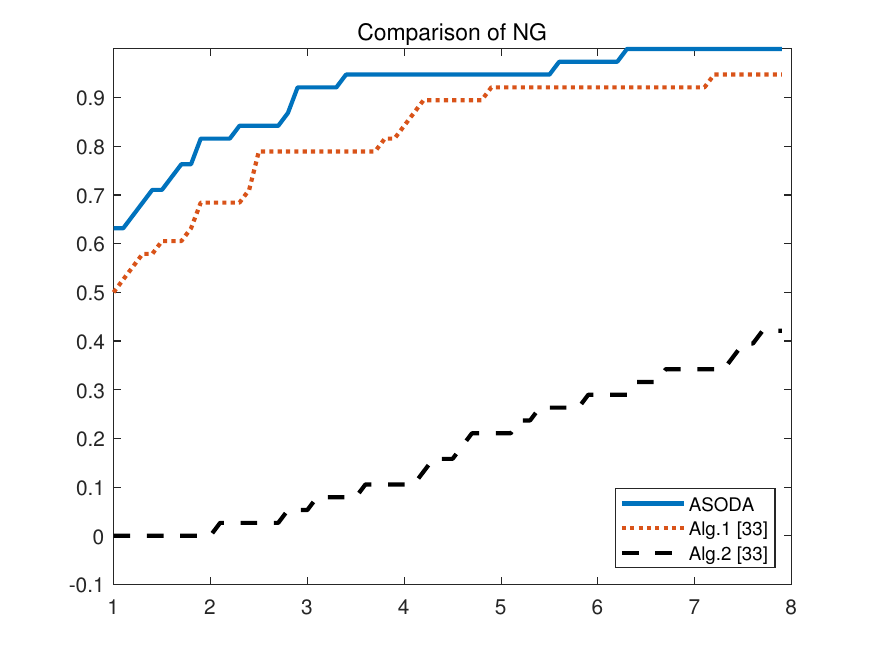}
\end{minipage}%
}%
\centering
\caption{Performance profiles comparing ASODA with Algorithms 1 and 2 in \cite{Arreckx} based on (a) function evaluations and (b) gradient evaluations. 
ASODA attains the fewest function evaluations on 33 of the 38 problems shown (see Table \ref{table6}).}\label{Figure2}
\end{figure}

The comparison results in Table \ref{table4}, Table \ref{table6} and Figures \ref{Figure1}--\ref{Figure2} further substantiate the efficiency of ASODA. Across the test problems, ASODA requires fewer function evaluations than its competitors in most cases: relative to Algorithm 2.1 in \cite{ZhongwenChen} it uses fewer, the same, and more evaluations on 24, 8 and 11 of the 43 problems in Table \ref{table4}, respectively; relative to LANCELOT, on 39, 2 and 1 of the 42 problems for which LANCELOT completed (the remaining problem, FLOSP2TL, is marked ``*''); relative to Algorithm 1 in \cite{Arreckx}, on 31, 2 and 5 of the 38 problems in Table \ref{table6}; and relative to Algorithm 2 in \cite{Arreckx}, it uses fewer evaluations on all 38 problems. These comparisons are made at the tolerance of the respective cited paper, as specified in the experimental setup above. The performance profiles show that ASODA attains the fewest function evaluations on 30 of the 42 problems of Table \ref{table4} and on 33 of the 38 problems of Table \ref{table6}, and its profile curve generally lies above those of the competitors over the range of $\tau$, indicating a robust advantage rather than a marginal one.

Two structural features of ASODA account for this advantage. First, ASODA performs exactly one evaluation of the function, the constraints, and the gradient per iteration---that is, NF$=$NG$=$NC$=$NI in Table \ref{table5.1}, and likewise NF$=$NC in Table \ref{table4} and NF$=$NG$=$NC in Table \ref{table6}---because each iteration solves a single GHM subproblem without any line search or additional merit-function probing. Line-search or trust-region SQP-type methods such as those in \cite{ZhongwenChen,Arreckx} evaluate the objective and the constraints repeatedly while adjusting the step, so that their evaluation counts are not tied to their iteration counts; this is a plausible structural reason why comparable iteration counts can translate into larger evaluation counts. Second, the adaptive cubic-regularization mechanism, together with the composite-step decomposition, produces steps of high quality, so that very few iterations are needed to reach the desired accuracy, and each such iteration is inexpensive in terms of derivative evaluations. Taken together, these observations indicate that the design choices---the homogenization-based adaptive second-order model and the composite-step SQP framework---translate into tangible practical gains, in line with the design goals set out in Sections \ref{sec2} and \ref{sec3}.

\emph{Scope of the comparison.} The present study compares ASODA with three established constrained solvers on a common CUTEst test set. A direct numerical comparison with the augmented-Lagrangian Newton--CG method of \cite{HeLuPong2023}---the closest competitor in terms of problem class---and with general-purpose solvers such as IPOPT and SNOPT is not included here, because these methods are naturally measured in different cost units (matrix--vector products and inner iterations versus factorization-based outer iterations; see the discussion in Section \ref{sec1}), so that a like-for-like comparison would require a careful, separate benchmarking effort. Extending the numerical study in this direction is left for future work.

\section{Conclusion}\label{sec:conclusion}

We have presented ASODA, a novel second-order algorithm for equality-constrained optimization that successfully integrates the homogeneous second-order descent framework with a composite-step SQP approach. The key achievement is a complexity analysis showing that ASODA attains the $\mathcal{O}(\epsilon^{-3/2})$ iteration bound of its unconstrained counterparts (HSODM/HSODF) under the same structural hypotheses, while every algorithmic parameter is a numerical constant independent of the problem data and of $\epsilon$. This theoretical guarantee is complemented by competitive numerical performance on standard test sets, showcasing the algorithm's practical potential.

Our analysis also includes a constructive treatment of the so-called ``hard case'', which confirms the algorithmic robustness. The current framework naturally suggests several promising avenues for future investigation. First, extending the methodology to general inequality constraints, potentially through active-set identification or interior-point techniques, would significantly broaden its applicability. Second, developing inexact and large-scale variants by incorporating iterative linear algebra (e.g., Krylov subspace methods) could unlock the algorithm's potential for very high-dimensional problems, as the current approach relies on explicit matrix factorizations. Finally, interfacing the method with automatic differentiation and state-of-the-art linear solvers would enhance its reproducibility and performance for a broader class of users, and a systematic study of the dependence of the iteration count on the tolerance $\epsilon$---which would empirically corroborate the $\mathcal{O}(\epsilon^{-3/2})$ bound---would complement the single-tolerance experiments reported here.
We also note that the present analysis relies on the full row rank of the constraint Jacobian $A_k$ at every iterate; extending the framework to rank-deficient constraints, for instance through weaker constraint qualifications or regularized normal steps, remains an open question for future work.

\begin{acknowledgements}
The authors gratefully acknowledge the support of the Natural Science Foundation of Henan Province (252300421993), the National Natural Science Foundation of China (12071133), and the Key Scientific Research Project for Colleges and Universities in Henan Province (25B110005).
\end{acknowledgements}

\section*{Declarations}

\noindent\textbf{Funding.} This work was supported by the Natural Science Foundation of Henan Province (252300421993), the National Natural Science Foundation of China (12071133), and the Key Scientific Research Project for Colleges and Universities in Henan Province (25B110005).

\noindent\textbf{Competing interests.} The authors declare that they have no conflict of interest.



\noindent\textbf{Author contributions.} All authors contributed to the conception and design of the study, the theoretical analysis, the numerical experiments, and the preparation of the manuscript. All authors read and approved the final manuscript.

\bibliographystyle{spmpsci}
\bibliography{SAODMrefs1_revised}

\end{document}